\documentclass{amsart}
\usepackage{latexsym,amsxtra,amscd,ifthen,amsmath,color, multicol,hyperref}
\usepackage{amsfonts}
\usepackage{verbatim}
\usepackage{amsmath}
\usepackage{amsthm}
\usepackage{amssymb}
\usepackage{tabu}
\usepackage[notcite,notref,final]{showkeys}
 \usepackage[all,cmtip]{xy}

\numberwithin{equation}{section}

\theoremstyle{plain}
\newtheorem{theorem}{Theorem}[section]
\newtheorem{lemma}[theorem]{Lemma}

\newtheorem{proposition}[theorem]{Proposition}

\newtheorem{corollary}[theorem]{Corollary}

\theoremstyle{definition}
\newtheorem{definition}[theorem]{Definition}
\newtheorem{example}[theorem]{Example}
\newtheorem{hypothesis}[theorem]{Hypothesis}
\newtheorem{notation}[theorem]{Notation}

\newtheorem{remark}[theorem]{Remark}

\newtheorem{problem}[theorem]{Problem}

\newtheorem{question}[theorem]{Question}

\makeatletter              
\let\c@equation\c@theorem  
\makeatother

\DeclareMathOperator{\gr}{gr}

\DeclareMathOperator{\op}{op}

\newcommand{\DOT}{\setlength{\unitlength}{1pt}\begin{picture}(2.5,2)
               (1,1)\put(2,2.5){\circle*{2}}\end{picture}}

\newcommand{\bu}{\DOT} 

\newcommand{\mc}{\mathcal}

\newcommand{\coh}{{\rm H}}

\newcommand{\N}{{\mathbb N}}
\newcommand{\id}{\mbox{\rm id}}

\newcommand{\ot}{\otimes}
\newcommand{\otb}{\otimes^{{\sf c}}}
\newcommand{\opc}{{\rm{op}}_{{\sf c}}}

\newcommand{\CC}{\mathbb{C}}

\newcommand{\beq}{\begin{equation}}
\newcommand{\eeq}{\end{equation}}

\begin{document}

\title[PBW deformations of braided products]
{PBW deformations of braided products}

\author{Chelsea Walton and Sarah Witherspoon}

\address{Department of Mathematics, Temple University, Philadelphia, Pennsylvania 19122
USA}

\email{notlaw@temple.edu}

\address{Department of Mathematics, Texas A\&M University, College Station, Texas 77843, USA}

\email{sjw@math.tamu.edu}

\bibliographystyle{abbrv}       

\begin{abstract}
We present new examples of deformations of  smash product algebras that arise  from Hopf algebra actions on  pairs of module algebras. These examples involve module algebras that are Koszul, in which case a PBW theorem we established previously applies. Our construction generalizes several `double' constructions appearing in the literature, including Weyl algebras and some types of Cherednik algebras, and it complements the braided double construction of Bazlov and Berenstein. Many suggestions of further directions are provided at the end of the work.
\end{abstract} 



\maketitle


\setcounter{section}{-1}


\section{Introduction}

Deformations and representations of `doubled' algebraic structures have been of great interest, especially in the last 15 years. Such deformations include Weyl algebras, rational Cherednik algebras (RCAs), double affine Hecke algebras (DAHAs) and generalizations. Here, we are interested in algebras $\mathcal D$ with a  {\it doubled structure} in the sense that $\mathcal D$ is isomorphic as a vector space to $A\ot H\ot B$ where 
\begin{enumerate}
\item[(i)] $H$ is a Hopf algebra that is a subalgebra of $\mathcal D$, 
\item[(ii)] $H$ acts on algebras $A, B$, and
\item[(iii)] $A$ and $B$ are compatible in some fashion (e.g. there is a {\it pairing} between $A$ and $B$).
\end{enumerate}
In this case, $A \otimes H \otimes B$ is referred to as a {\it triangular decomposition} or a {\it Poincar\'{e}-Birkhoff-Witt} ({\it PBW}) {\it decomposition} of $\mathcal D$. Indeed, all of this is modeled on the decomposition of the universal enveloping algebra  $U(\mathfrak{g})$ of a finite dimensional semisimple Lie algebra $\mathfrak g$ as a tensor product $U(\mathfrak{n}^-) \otimes U(\mathfrak{h}) \otimes U(\mathfrak{n}^+)$ via the classical PBW theorem, where $\mathfrak{n}^- \oplus \mathfrak{h} \oplus \mathfrak{n}^+$ is the triangular decomposition of $\mathfrak{g}$. Consider Example~\ref{ex:rCa} and the table below for the examples mentioned above, and see \cite[\S2.12]{Che}, \cite{Dr, Et, EG, Rou}
for further reading.

\bigskip

{\small
\begin{center}
\begin{tabular}{|l||l|l|l|}
\hline
$\mathcal D$ & $H$ & $A$ & $B$\\
\hline
\hline
Weyl algebra & a field $k$ & $k[x_1, \dots, x_n]$ & $k[y_1, \dots, y_n]$, $y_i = \frac{\partial}{\partial x_i}$ \\
RCA = rational (degenerate) DAHA & 
$\mathbb{C}\Gamma$, $\Gamma \leq GL_n(\mathbb{C})$ cpx. ref. group &  $\mathbb{C}[x_1, \dots, x_n]$ & $\mathbb{C}[y_1, \dots, y_n]$\\
trigonometric (degenerate)  DAHA & $\mathbb{C} \Gamma $ &  $\mathbb{C}[x_1, \dots, x_n]$ & $\mathbb{C}[y_1^{\pm 1}, \dots, y_n^{\pm 1}]$\\
\hline
\end{tabular}
\end{center}
}
\medskip

\begin{center}
\textsc{Table 1.} Examples of `doubled' algebras
\end{center}
\bigskip

 The goal of this paper is to present generalizations of doubled structures that have not appeared in the literature, particularly in cases where $H$ is non-cocommutative or $A$ and $B$ are non-commutative Koszul algebras. Recall that an $\N$-graded algebra $A$ with $A_0=k$ is {\it Koszul} if its trivial module  $A/A_{>0}\cong k$ admits a linear graded free resolution; see \cite[Chapter~2]{PP}.
\smallskip
  
We achieve our aim by working  primarily in categories of $H$-modules that are braided, and by using the braided tensor product $\otb$ of $H$-module algebras considered by
Manin \cite[Chapters~11--12]{Manin:QGNG} and Baez~\cite{Baez}, to form the desired smash product algebra $(A \otb B) \# H$. Then, we apply our previous work \cite{WW} to compute PBW deformations of these smash product algebras.
This method is convenient since a braided product of two Koszul $H$-module algebras is again a Koszul $H$-module algebra (see Corollary~\ref{cor:twistkoszul}).
However, this construction can be somewhat limited for some types of algebras; for example, the smash product algebras arising in this manner may not admit non-trivial PBW deformations, or these deformations may be difficult to compute. 
We expand the construction by generalizing our method to apply to a twisted tensor product $\otimes^\tau$ of $H$-module algebras studied by {\v{C}}ap, Schichl, and Ven{\v{z}}ura \cite{CSV}. The twisted tensor product $A\ot^{\tau} B$ of two Koszul algebras $A$ and $B$ is known to be Koszul (see \cite[Corollary 4.19]{CSV} or Proposition~\ref{lem:twistKoszul} below), and if the twisting map $\tau$ is an $H$-module homomorphism, then $A\ot^{\tau} B$ is also an $H$-module algebra. For these examples, we may form the smash product algebra $(A \otimes^\tau B) \# H$ and compute its PBW deformations.
\smallskip

Definitions and preliminary results on braided products and on twisted tensor
products of $H$-module algebras are presented in Section~\ref{sec:braided}, and our main theorem from~\cite{WW} on PBW deformations of smash product algebras arising from Hopf actions on Koszul algebras  is recalled in Section~\ref{sec:WW}.
\smallskip

Before we present the wealth of new examples of PBW deformations constructed here, we compare our algebras with the {\it braided doubles} of Bazlov and Berenstein \cite{BB1,BB2}. This is summarized in Table~2 below.  
\medskip

\noindent {\bf Notation.} We work over an arbitrary field $k$. An unadorned $\otimes$ will mean $\otimes_k$, and is sometimes suppressed without mention. In some sections, we will put a restriction on the characteristic of $k$.
  
For a Hopf algebra $H$, let $S(V)$ (resp., $S_{\bf q}(V)$) denote the symmetric
(resp., ${\bf q}$-symmetric) algebra on a finite-dimensional $H$-module $V$. 
Here, ${\bf q} \in \text{End}_H(V\otimes V)$.  Let $T(V)=T_k(V)$ be the free $k$-algebra on $V$, and $( I )$ denote the ideal of $T(V)$ generated by~a set~$I \subset T(V)_{>0}$. 

\medskip

We see that our work is orthogonal to that of Bazlov and Berenstein in the sense that they do not require  $A\ot B$ itself to be an $H$-module algebra, but elements of $A$ and $B$ pairwise commute (or skew commute), while we require $A\ot B$ to be an $H$-module algebra, but allow more general relations between elements of $A$ and elements of $B$. 
The deformations of the products of $H$-module algebras in this work are thus of a somewhat different form than in \cite{BB1,BB2}, although there is significant common ground. Like Bazlov and Berenstein, we recover rational Cherednik algebras via our construction; in fact, a direction for research pertaining to  this example is prompted by \cite[Theorem~G]{BB1} (see Question~\ref{q:RCA}).
\bigskip

{\small
\begin{center}
{
\tabulinesep=0.65mm
\begin{tabu}{|c||c|c|}
\hline
{\bf Conditions} &{\bf This work} (using braiding $\sf c$ or twisting $\tau$) & {\bf Bazlov-Berenstein} \cite{BB1,BB2}\\
\hline
\hline
On $H$ & 
\begin{tabular}{l}
using $\sf c$: (sub)category $H$-mod is braided \\
using $\tau$: --none-- 
\end{tabular}&
No conditions\\
 \hline
On $A$ & 
 Is a Koszul $H$-module algebra & 
 \begin{tabular}{l}
 Is denoted $T(V)/ ( I^- )$, No conditions\\
 \hspace{.1in} e.g. $S(V)$ or $\mathfrak{B}(V)$ (resp.,  or $S_{\bf q}(V)$)
 \end{tabular}\\
  \hline
On $B$ & 
\begin{tabular}{l}
 Is a Koszul $H$-module algebra\\
 \begin{tabular}{l}
using $\sf c$:  e.g. {\it braided-opposite} $A^{\text op}_c$   of  $A$ \\
  \hspace{.43in} (see Proposition~\ref{lem:HopKoszul})\\
using $\tau$:  $B$ so that $\tau$ is an $H$-mod map
\end{tabular}
\end{tabular}
&
 \begin{tabular}{l}
 Is denoted $T(V^*)/ ( I^+ )$ so that\\
\hspace{.1in} $T(V^*)/ ( I^+ )$ and $T(V)/ ( I^- )$ satisfy a\\
 \hspace{.14in}{\it non-deg.  \hspace{-.1in} Harish-Chandra pairing}\\
 e.g. $S(V^*)$ or $\mathfrak{B}(V^*)$ (resp.,  or $S_{\bf q}(V^*)$)
 \end{tabular}\\
 \hline
 \begin{tabular}{c}
On product\\
   of $A$ and $B$
   \end{tabular} & 
   \begin{tabular}{l}
  \hspace{-.08in} ``Mixed" relations are derived naturally\\
     \hspace{.19in} from braiding/twisting \\
  Is a Koszul $H$-module algebra \\
    \begin{tabular}{l}
   \hspace{.1in}- $H$-mod alg by Remarks~\ref{rk:braided},  \ref{rk:Hmod-tau}\\
  \hspace{.1in}- Koszul by  Prop~\ref{lem:twistKoszul}, Cor \ref{cor:twistkoszul}
   \end{tabular}
   \end{tabular}
    &
    \begin{tabular}{l}
    \hspace{-.15in}``Mixed" relations are prescribed: \\
    \hspace{.1in} $[f,v] = 0$ (resp., $[f,v]_{\bf q} = 0$) \\ 
   \hspace{.1in}  for $f \in V^*$ and $v \in V$\\
    \hspace{-.15in} Is not always an $H$-module algebra 
     \end{tabular}
\\
\hline
\begin{tabular}{c}
On deformations\\ 
of relations\\
 of $A$ or $B$ 
 \end{tabular}
 &
 \begin{tabular}{c}
 Apply Theorem~\ref{thm:mainWW} (\cite{WW}) to deform\\
relations by elements  of degree 0 or 1 
 \end{tabular}
 & Relations are not deformed \\
\hline
\begin{tabular}{c}
On deformations\\ 
of ``mixed" relations \\
of product $A$ and $B$ 
 \end{tabular}
 &
 \begin{tabular}{c}
 Apply Theorem~\ref{thm:mainWW} (\cite{WW}) to deform\\
relations by elements  of degree 0 or 1 
 \end{tabular}
 & 
 \begin{tabular}{l}
Deform by elements  of degree 0 only: \\
 for $k$-linear map $\beta: V^* \otimes V \to H$, get\\
 $[f,v] = \beta(f,v)$ (resp., $[f,v]_{\bf q} = \beta(f,v)$)
 \end{tabular}
 \\
\hline
\end{tabu}
}
\end{center}
}

\bigskip

\begin{center}
\textsc{Table 2.} Comparing our  {\it braided/twisted products} to Bazlov and Berenstein's  {\it braided doubles}
\end{center}

\bigskip

Our main result is a  generalization of the `doubled' algebras  mentioned at the beginning of this section, complementing the braided doubles of Bazlov and Berenstein discussed above. By a {\it PBW deformation of degree 0} (or a {\it degree 0 PBW deformation}), we mean that  $\kappa = \kappa^C$ in Hypothesis~\ref{hypothesis} in Section~\ref{sec:WW}.

\begin{theorem} 
Given certain Hopf algebras $H$ and Koszul algebras $A,B$ as listed in Table~3, we find PBW deformations $\mathcal D$ of degree 0 of the smash product algebra $(A \otimes^* B) \#H$, where $\otimes^*$ is either
\begin{itemize}
\item  a braided product $\otb$ (in the case when a category of $H$-modules is braided), or \item a twisted tensor product $\otimes^\tau$ (in general).
 \end{itemize} 
Here, either the Hopf algebra $H$ is non-cocommutative or the Koszul algebras $A,B$ are noncommutative. 
The parameter space of \underline{all} of such PBW deformations is computed in the cases denoted by $\bigstar$ in Table~3.
\end{theorem}

\medskip

{\small
\begin{center}
{
\tabulinesep=.65mm
\begin{tabu}{|l||c|c|c|c|c|}
\hline
{\bf Section} &
 $H$ & $A$ & $B$ &  \begin{tabular}{c}
{\bf Braiding}/
\\
{\bf Twisting}
 \end{tabular} &
  \begin{tabular}{c}
    {\bf Parameter space of}\\ 
   {\bf degree 0 PBW deformations}\\
   {\bf of} $(A\otimes^* B) \# H$
   \end{tabular}
 \\
\hline
\hline
\ref{uqsl2braid}
 & 
$U_q(\mathfrak{sl}_2)$
 &
$k_q[u,v]$
  &
$\displaystyle A^\opc = k_q[u,v]$
&
braiding $\sf c$ \eqref{braiduqsl2}
&
\begin{tabular}{l}
 $k \times \mathbb{Z}^{ 3} \times \mathbb{N}^{ 4}$
\end{tabular}
\\
\hline
\ref{subsec:tau}
 & 
$U_q(\mathfrak{sl}_2)$
& 
$k_q[u,v]$
&
$ k_q[u,v]$
&
twisting $\tau$ \eqref{twistuqsl2}
 &
 \begin{tabular}{l}
 $k \times \mathbb{Z}$  
\end{tabular}
\\
\hline
\ref{sec:k[u,v]} $\bigstar$  
& 
$T(2)$
& 
$k[u,v]$ 
&
$\displaystyle A^\opc = k[u,v]$
&
\begin{tabular}{c}
braiding $\sf c$ with\\
$\text{R}$-matrix \eqref{T2Rmatrix} 
\end{tabular}
 &
 $k$
\\
\hline
 \ref{kJRC2} $\bigstar$
& 
$kC_2$
& 
$k_J[u,v]$
& 
$A^\opc \cong k_J[u,v]$
& 
\begin{tabular}{c}
braiding $\sf c$ with\\
$\text{R}$-matrix \eqref{RforC2}
\end{tabular}
 &
 \begin{tabular}{l}
 $k^{ 3}$
\end{tabular}
\\
\hline
 \ref{kJtrivR} $\bigstar$
& 
$kC_2$
& 
$k_J[u,v]$
& 
$A^\opc \cong k_J[u,v]$
& 
\begin{tabular}{c}
braiding $\sf c$ with\\
$\text{R}= 1\otimes 1$
\end{tabular}
&
\begin{tabular}{l}
 $k^{3}$
\end{tabular}
\\
\hline
 \ref{sec:Ska} $\bigstar$
& 
$kC_2$
& 
$S(a,b,c)$
& 
$A^\opc =S(b,a,c)$
& 
\begin{tabular}{c}
braiding $\sf c$ with\\
$\text{R}$-matrix \eqref{RforC2}
\end{tabular}
&
\begin{tabular}{l}
 $k^{ 6}$ ~~ if $a\neq b$\\
 $k^{ 15}$ if $a=b$
\end{tabular}
\\
\hline
 \ref{sec:Skb} $\bigstar$
& 
$kC_2$
& 
$S(a,b,c)$
& 
$A^\opc =S(b,a,c)$
& 
\begin{tabular}{c}
braiding $\sf c$ with\\
$\text{R} = 1 \otimes 1$
\end{tabular}
 &
\begin{tabular}{l}
 $k^{ 6}$ ~~ if $a\neq b$\\
 $k^{ 15}$ if $a=b$
\end{tabular}\\
\hline 
\end{tabu}
}
\end{center}
}

\medskip

\begin{center}
\textsc{Table 3.} Main result: Parameterization of PBW deformations of braided (or twisted tensor) products\\
\end{center}

\medskip

In the first set of examples (Section~\ref{sec:U kq}), we compute some degree~0 PBW deformations  arising from the action of the quantized enveloping algebra $U_q(\mathfrak{sl}_2)$ on the quantum plane $k_q[u,v]$ when $q$ is a root of unity. In fact, the results in Section~\ref{subsec:tau} can be extended to examples involving an action of  $U_q(\mathfrak{gl}_2)$ (see Remark~\ref{rk:uqgl2}). In Section~\ref{sec:k[u,v]}, the Sweedler Hopf algebra $T(2)$ acts on the plane, and we find all PBW deformations of degree~0. Lastly, we find all degree~0 PBW deformations when $A$ is  either the Jordan plane $k_J[u,v]$ or a Sklyanin algebra $S(a,b,c)$ (Sections~\ref{sec:Jordan}, \ref{sec:Sk}), with an action of a cyclic group of order 2, for both the trivial braiding and a nontrivial braiding of $H$-mod.
\smallskip

Open questions and further directions of this study are presented in Section~\ref{directions}.


\section{Braided products}\label{sec:braided}

We will need the notion of a braided tensor product of Hopf module algebras; see, e.g., Manin \cite[Chapters~11--12]{Manin:QGNG},  Baez~\cite{Baez}, or Majid \cite{Majid}.
We include some details for completeness.
For background on braided monoidal categories, see e.g.\ Baez~\cite{Baez} or 
Bakalov and Kirillov~\cite{BK}, or Kassel \cite[Sections~13.1--13.3]{Kassel}.

Consider the following notation/hypotheses:

\begin{hypothesis}[$H$, $\mathcal{C}$]
Let $H$ be a Hopf algebra over a field $k$, with standard structure notation: $(H, m, \Delta, u ,\epsilon, S)$, and let $\mathcal C$
be a monoidal category of (left) $H$-modules.

Assume that $\mathcal C$ comes equipped with a braiding, that is, 
there are functorial isomorphisms
$$
    {\sf c}_{M,N}: M\ot N \stackrel{\sim}{\relbar\joinrel\longrightarrow}
     N\ot M
$$
for all pairs of objects $M,N$ in $\mathcal C$, satisfying the hexagon axioms.
(See, for example, \cite[Definition~XIII.1.1]{Kassel}.) 
\end{hypothesis}

\begin{example}
If $H=U_q({\mathfrak g})$, we could 
take $\mathcal C$ to be the category of locally finite-dimensional $H$-modules. 
In general, if $H$ is quasitriangular, we may take ${\sf c}_{M,N}$ to be given by the
action of an $\text{R}$-matrix. 
\end{example}

Now we define braided products and braided opposites of $H$-module algebras in $\mathcal{C}$.
Recall that an {\em $H$-module algebra} is an algebra $A$ that is an $H$-module in such
a way that $h\cdot (aa') = \sum (h_1\cdot a)(h_2\cdot a')$ and $h\cdot 1_A=\epsilon(h)1_A$
for all $h\in H$ and $a,a'\in A$.  Here, we employ Sweedler's notation: $\Delta(h) = \sum h_1 \otimes h_2$.

\begin{definition}[$A \otb B$, $A^\opc$] Let $A,B$ be two (left) $H$-module algebras that are in category $\mathcal C$. 
\begin{enumerate}
\item The {\em braided product} $A\otb B$
is $A\ot B$ as an object of $\mathcal C$,
and multiplication is defined by the composition
\begin{displaymath}
\xymatrix{
    A\ot B\ot A\ot B \ar[rr]^{1\ot {\sf c}\ot 1}&& A\ot A\ot B\ot B \ar[rr]^{m_A\ot m_B} 
        && A\ot B} , 
\end{displaymath}
where ${\sf c}={\sf c}_{A,B}$ and 
$m_A$ and $m_B$ are multiplication on $A$ and $B$, respectively. This multiplication is indeed associative; see e.g.\ \cite[Lemma 2]{Baez} or \cite[Lemma 2.1]{Majid}.
\item The {\em braided-opposite algebra} of the (left) $H$-module algebra $A$,
denoted $A^\opc$, 
is $A$ as an object of $\mathcal C$ and multiplication is defined by
\begin{displaymath}
\xymatrix{
    A\ot A \ar[r]^{{\sf c}}& A\ot A \ar[r]^{m_A} 
        & A ,} 
\end{displaymath}
where ${\sf c} = {\sf c}_{A,A}$. Associativity of multiplication is proved in \cite[Lemma 1]{Baez}. 
\end{enumerate}
\end{definition}

\begin{remark}\label{rk:braided}
The braided product $A\otb B$ is again an $H$-module algebra, as the following
commutative diagram shows. Let $h\in H$. Consider the diagram: 
\begin{displaymath} 
  \xymatrix{ 
  A\ot B\ot A\ot B\ar[d]^{h\cdot} \ar[rr]^{1\ot {\sf c}\ot 1} &&
       A\ot A\ot B\ot B \ar[d]^{h\cdot } \ar[rr]^{m_A\ot m_B} && A\ot B\ar[d]^{h\cdot} \\
   A\ot B\ot A\ot B \ar[rr]^{1\ot {\sf c}\ot 1} && A\ot A\ot B\ot B \ar[rr]^{m_A\ot m_B}
    && A\ot B }
\end{displaymath}
The left square commutes because $\sf c$ is an $H$-module homomorphism. 
The right square commutes because multiplication in $A$ and in $B$ are $H$-module 
homomorphisms. 

The braided-opposite algebra $A^\opc$ is again an $H$-module algebra, via the 
original action of $H$ on $A$, since $\sf c$ is an $H$-module homomorphism. 
\end{remark}

Recall that if $A$ is an $H$-module algebra, we may form the 
{\em smash product algebra} $ A \# H$, that is
$A\ot H$ as a vector space, with multiplication $(a\ot h)(a'\ot h')
=\sum a (h_1\cdot a')\ot h_2 h'$ for all $a,a'\in A$ and $h,h'\in H$. 
We will write $a\# h$ or more simply $ah$ for the element $a\ot h$
in $A\# H$ when no confusion will arise. 
We will be interested in $H$-module algebras of the form $A\otb B$,
often in the case where $B = A^\opc$, and deformations of
the resulting smash product algebra $(A\otb B)\# H$. 

\subsection{Twisted tensor products and Koszulity}

When considering `doubles' of $H$-module algebras, say of $A$ and $A^\opc$, one advantage of using a braided product  is that $A\otb A^\opc$ is automatically an $H$-module algebra, as we saw in Remark~\ref{rk:braided}. However, a disadvantage is that the supply of braidings may be limited, or that deformations of smash product algebras may be difficult to compute. One way to remedy this is to consider a more general product of two ($H$-module) algebras $A$ and $B$: the twisted tensor product.

\begin{definition}[$A \otimes^\tau B$] \cite{CSV}
Let $A$ and $B$ be algebras over $k$. A {\it twisted tensor product} $A\otimes^\tau B$ of $A$ and $B$ is the $k$-vector space $A\otimes B$, with multiplication $m_\tau$ defined as follows. Let $$\tau: B \ot A \to A \ot B$$ be a $k$-linear mapping for which $\tau(b\ot 1) = 1 \ot b$ and $\tau(1 \ot a) = a \ot 1$ for all $b\in B$
and $a\in A$. Take 
$$m_\tau:= (m_A \ot m_B) \circ (id_A \ot \tau \ot id_B),$$
and consider the associativity constraint
\begin{equation}\label{eqn:assoc}
\tau \circ (m_B \ot m_A) = m_\tau \circ (\tau \otimes \tau) \circ (id_B \ot \tau \ot id_A)
\end{equation}
as maps from $B\ot B\ot A\ot A$ to $A\ot B$. 
If $\tau$ satisfies this constraint~(\ref{eqn:assoc}), we call $\tau$ a {\em twisting map}, and in this case, $A\ot^{\tau}B$ is an associative algebra~\cite[Proposition/Definition~2.3]{CSV}. 
If $A$ and $B$ are graded algebras, we say that $\tau$ is {\em graded} if
$\tau(B_j\ot A_i)\subseteq A_i\ot B_j$ for all $i,j$. 
\end{definition}

\begin{remark} \label{rk:Hmod-tau} If $A$ and $B$ are $H$-module algebras, and if $\tau$ is an $H$-module twisting map, then the twisted tensor product $A \otimes^\tau B$ is an $H$-module algebra by the same reasoning as in Remark~\ref{rk:braided}. 
\end{remark}

Koszulity is also preserved, as the following proposition shows. 
We include a proof for completeness, although the result is known.
See, for example, \cite[Corollary 4.19]{JPS} or \cite[p.\ 90, Example 3]{PP}.

\begin{proposition} \label{lem:twistKoszul}
Assume that $A$ and $B$ are graded Koszul $H$-module algebras, and that $\tau:B\ot A\rightarrow A\ot B$ is a graded twisting map that is also an $H$-module homomorphism.
Then the twisted tensor product $A \otimes^\tau B$ is a graded Koszul $H$-module algebra.
\end{proposition}

\begin{proof}
Since  $A$ and $B$ are Koszul, we may write
$A=T(V)/(I)$ and $B=T(W)/(J)$,
for vector spaces $V$ and $W$, where $I\subseteq V\ot V$ and $J\subseteq W\ot W$.
The Koszul resolution of $k$ as an $A$-module may be expressed as $K_{\bu}(A)=\oplus_{n\geq 0}K_n(A)$
where 
$$
K_0(A) =A, \quad K_1(A) = A \otimes V, \quad  \text{ and } \ 
K_n(A) = A\ot \bigcap_{i+j=n-2} (V^{\ot i}\ot I\ot V^{\ot j})  \ \text{ for $n \geq 2$.}
$$
The differentials are those induced by the canonical embedding into the
bar resolution $B_{\bu}(A)$ of $k$.
(Recall that this bar resolution  is defined by  $B_n(A)=A^{\ot (n+1)}$, with differentials
$$
   \delta_n(a_0\ot\cdots \ot a_n) = (-1)^n\epsilon(a_n)a_0\ot \cdots \ot a_{n-1} +
   \sum_{i=0}^{n-1} (-1)^i a_0\ot \cdots\ot a_ia_{i+1}\ot \cdots\ot a_n
$$
for all $a_0,\ldots, a_n\in A$.) 
Let $K_{\bu}(B)$ denote the Koszul resolution of $k$ as a $B$-module, defined similarly. 

We may use the map $\tau: B\ot A\rightarrow A\ot B$ to define a map
$\tau_n: B\ot A^{\ot n}\rightarrow A^{\ot n}\ot B$ iteratively:
Let $\tau_1=\tau$, $\tau_2 =(\id_A\ot \tau)\circ (\tau\ot \id_A)$, and so on.  
We claim that $\tau_{n+1}$ sends $B\ot K_n(A)$
to $K_n(A)\ot B$. The claim is true if $n=0$ or $n=1$, as $\tau$ is graded.
If $n=2$, we have $K_2(A)= A\ot I$. Since $I$ is the kernel of 
$m_A | _{V\ot V}$, we will see that the associativity constraint (\ref{eqn:assoc}) ensures that
$\tau_3$ sends $B\ot A\ot I$ to $A\ot I\ot B$ as follows. 
First note it suffices to show that 
$(\id_A\ot \id_A\ot \tau)\circ (\id_A\ot \tau\ot \id_A)$ 
takes $k\ot B\ot I$ to $k\ot I\ot B$
(as the application of $\tau\ot\id_A\ot\id_A$ to $B\ot A\ot I$, sending it to
$A\ot B\ot I$, does not affect the tensor factor $I$). 
Since $\tau(1\ot a) = a\ot 1$ for all $a\in A$, this is equivalent to the statement that
the following map takes $k\ot B\ot I$ to $I\ot k\ot B$ (identifying $k$
with $B_0$ here):
$$
 (\id_A\ot \tau\ot \id_B)(\tau\ot \id_A\ot \id_B) (\id_B\ot \id_A\ot \tau) (\id_B\ot \tau\ot \id_A)
  = (\id_A\ot \tau\ot \id_B)(\tau\ot \tau) (\id_B\ot \tau\ot\id_A) . 
$$
Compose this map with $m_A\ot m_B$. By~(\ref{eqn:assoc}) and the definition of $m_{\tau}$, we obtain 
$\tau\circ (m_B\ot m_A)$ as a map from $k\ot B\ot I$ to $A\ot B$. Since $m_A$ takes $I$ to 0,
by retracing our steps, we see that $m_A\ot m_B$ takes the image of 
$$
   (\id_A\ot \tau\ot\id_B) (\tau\ot\tau) (\id_B\ot \tau\ot\id_A)
$$
on $k\ot B\ot I$ to 0. By canonically identifying $k\ot B\ot I$ with $B\ot I$, we see that this
implies that $m_A\ot \id_B$ takes the image of $(\id_A\ot\tau) (\tau\ot \id_A)$ on $B\ot I$ to 0.
Since $I$ is precisely the kernel of $m_A$ on $V\ot V$, and since $\tau$ is graded, this implies
that the image of $(\id_A\ot \tau) (\tau\ot \id_A)$ on $B\ot I$ is indeed $I\ot B$. 
If $n>2$, an inductive argument shows that $\tau_{n+1}$ sends $B\ot K_n(A)$
to $K_n(A)\ot B$. 

We will take a twisted tensor product of the Koszul resolutions $K_{\bu}(A)$ and $K_{\bu}(B)$ to
form a linear graded free resolution of $k$ as $A\ot^{\tau} B$-module, thus proving that
$A\ot^{\tau}B$ is Koszul. Let
$$
   K_{\bu}(A\ot ^{\tau} B ) = \bigoplus_{n\geq 0} K_n(A\ot ^{\tau}B) ,
$$
where $K_n(A\ot ^{\tau}B) = \oplus_{i+j=n} (K_i(A)\ot K_j(B))$, that is,
as a complex of vector spaces, we take $K_{\bu}(A\ot^{\tau}B)$ to be the tensor product, over $k$,
of $K_{\bu}(A)$ and $K_{\bu}(B)$. By the K\"unneth Theorem, since the tensor product is taken over
the field $k$, the tensor product of these two complexes is acyclic with 
$\coh_0(K_{\bu}(A)\ot K_{\bu}(B)) \cong \coh_0(K_{\bu}(A))\ot \coh_0(K_{\bu}(B)) \cong k\ot k \cong k$.
By the definition of the differential on a tensor product, it will be linear since the
differentials on $K_{\bu}(A)$ and $K_{\bu}(B)$ are linear. 
We must put the structure of an $A\ot^{\tau}B$-module on each $K_n(A\ot^{\tau}B)$ in such a
way that the differentials are module homomorphisms. 
One may simply apply the twist $\tau$ iteratively to all tensor factors below to obtain an
induced map
$$
  (A\ot ^{\tau}B) \ot K_{\bu}(A)\ot K_{\bu}(B) \stackrel{\tau_{\bu}}{\longrightarrow}
   (A\ot K_{\bu}(A))\ot (B\ot K_{\bu}(B))  \stackrel{\rho_{K_{\bu}(A)}\ot \rho_{K_{\bu}(B)}}
  {\relbar\joinrel\relbar\joinrel\relbar\joinrel\relbar\joinrel\relbar\joinrel 
   \relbar\joinrel\relbar\joinrel\relbar\joinrel\longrightarrow} K_{\bu}(A)\ot K_{\bu}(B). 
$$ 
Here, $\rho_{K_{\bu}(A)}$ is the $A$-module structure map on $K_n(A)$, for all $n \geq 0$, and
similarly for $B$. 
By the associativity constraint~(\ref{eqn:assoc}), this gives each $K_n(A\ot^{\tau} B)$ the structure
of an $A\ot^{\tau} B$-module, and the differentials are module homomorphisms. 
\end{proof} 

\subsection{On Koszulity of braided-opposite algebras and of braided products}

Now we return to the setting of braided monoidal categories of $H$-modules. 
The following result is a consequence of Proposition~\ref{lem:twistKoszul} above.

\begin{corollary} \label{cor:twistkoszul}Take a Hopf algebra $H$ and a braided monoidal category
$\mathcal{C}$ of $H$-modules, say with braiding $\sf c$. If $A$ and $B$ are graded Koszul $H$-module algebras, then the braided product $A \otb B$ is also a graded Koszul $H$-module algebra in
$\mathcal{C}$. 
\end{corollary}

\begin{proof}
We can take $\tau = {\sf c}_{B,A}$, and apply Proposition~\ref{lem:twistKoszul}.
\end{proof}

Recall that we will often take $B$ to be the braided-opposite algebra $A^\opc$; we consider its Koszulity below.

\begin{proposition} \label{lem:HopKoszul}Take a Hopf algebra $H$ and a braided monoidal category $\mathcal{C}$
of $H$-modules,  with braiding $\sf c$. If $A$ is a graded Koszul $H$-module algebra in $\mathcal{C}$ and $H$ preserves the grading of $A$, then the braided-opposite algebra $A^\opc$ is also a graded Koszul $H$-module algebra. 
\end{proposition}

\begin{proof}
We have already seen that $A^\opc$ is an $H$-module algebra in Remark~\ref{rk:braided}. 
Let $A^e_{\sf c}:= A\otb A^\opc$. We generalize \cite[Proposition 19 and Corollary 8]{Kraehmer} 
to the braided setting. For each nonnegative integer $n$, let 
$$
   K'_n(A^e_{\sf c},A)= \bigcap_{i+j=n} V^{\ot i}\ot I\ot V^{\ot j} \quad \mbox{ and } \quad 
  K_n(A^e_{\sf c},A):= A\ot K'_n(A^e_{\sf c}, A)\ot A 
$$
as vector spaces. Then $K_n(A^e_{\sf c},A)$ is a 
left $A^e_{\sf c}$-module, where 
$$
  (a\ot b)\cdot (a'\ot x\ot b') =
  (m_A\ot 1 \ot m_A) (1\ot {\sf c}_{A, A\ot K_n'\ot A} ) (a\ot b\ot a'\ot x\ot b'),
$$
for all $a, b, a', b' \in A$ and $x \in K'_n$. Here, we extended the braiding $\sf c$ of $A \ot A$ iteratively to $A \ot K_n(A)$, as in the proof of Proposition~\ref{lem:twistKoszul}. We define the differentials of the complex $K_{\bu}(A^e_{\sf c},A)$ to be those  induced by its embedding into the {\em braided bar resolution}  of $A$ as an
$A^e_{\sf c}$-module defined by Baez in~\cite[Section~3]{Baez}.
The braided bar resolution is almost the same as the bar resolution of $A$ as an $A$-bimodule,
that is, the terms and differentials are the same, however we replace the usual $A$-bimodule structure
of the terms with the $A^e_{\sf c}$-module structure as just described. 

Now apply $\ - \ot_A k$ to $K_{\bu}(A^e_{\sf c},A)$ to obtain
$$
  K_{\bu}(A^e_{\sf c},A)\ot_A k \cong K_{\bu}(A,k),
$$
the Koszul resolution of $k$ as an $A$-module.
By hypothesis, $K_{\bu}(A,k)$ is acyclic, i.e.\ $\coh_n(K_{\bu}(A^e_{\sf c},A)\ot_A k)
\cong \coh_n(K_{\bu}(A,k)) =0$ for all $n>0$. 
We adapt the argument in \cite[proof of Proposition~19]{Kraehmer} as follows. The K\"unneth Theorem applies
since $K_{\bu}(A^e_{\sf c},A)$ consists of free right $A$-modules and the
standard contracting homotopy is a right $A$-module homomorphism. 
As a consequence, 
$$
  \coh_n(K_{\bu}(A^e_{\sf c},A))\ot_A k \cong \coh_n(K_{\bu}(A^e_{\sf c}, A)\ot_A k)  =0
$$
for all $n > 0$, which implies  $\coh_n(K_{\bu}(A^e_{\sf c},A)) =
\coh_n(K_{\bu}(A^e_{\sf c},A))A_{>0}$.
Since $\coh_n(K_{\bu}(A^e_{\sf c},A))$ is an $A$-module with grading inherited
from that of $A$, and action by $A_{>0}$ increases degree, this implies 
$\coh_n(K_{\bu}(A^e_{\sf c},A))=0$ for $n >0$. So, $K_{\bu}(A^e_{\sf c},A)$ is acyclic. Similarly, we may now apply $k\ot_A - \ $
to $K_{\bu}(A^e_{\sf c},A)$ to obtain
$$
   k\ot _A (A\ot K'_{\bu}(A)\ot A) \cong K'_{\bu}(A)\ot A,
$$
which is a resolution of right $A$-modules, equivalently 
left $A^\opc$-modules; that it is acyclic follows
the same reasoning as above. By construction, $K'_{\bu}(A)\ot A$ is a linear free
resolution of $k$ as an $A^\opc$-module, and therefore $A^\opc$ is Koszul.
\end{proof}

\medskip


\section{Standing Hypotheses and Recollections from \cite{WW}}\label{sec:WW}

In this section, we recall terminology, hypotheses, and results from \cite{WW} that we will need.

 \begin{definition} \label{def:PBW}
Let ${\mathcal D}=\bigcup_{i \geq 0} F_i$ be a filtered algebra with $\{0\} \subseteq F_0 \subseteq F_1 \subseteq \cdots \subseteq {\mathcal D}$. We say that $\mathcal D$ is a {\it Poincar\'e-Birkhoff-Witt (PBW) deformation} of an $\N$-graded algebra $R$ if  the associated graded  algebra $\gr_F {\mathcal D} = \bigoplus_{i \geq 0} F_i/F_{i-1}$ is isomorphic to $R$ as an $\N$-graded algebra.
\end{definition}

\begin{hypothesis} [$V, R, I, \kappa, \kappa^C, \kappa^L$] 
\label{hypothesis}  First, let $V$ be a finite dimensional $H$-module. 
\begin{itemize}
\item Let $I\subseteq V\ot V$ be an $H$-submodule for which 
$R := T(V)/(I)$ is an $\mathbb{N}$-graded  Koszul algebra, $R = \bigoplus_{j \geq 0} R_j$, with $R_0 =k$. In particular, $R$ is an $H$-module algebra, so that the $H$-action preserves the grading of $R$.
\smallskip

\item Take $\kappa:I \rightarrow H \oplus (V \otimes H)$ to be a $k$-linear map,
where $\kappa$ is the sum of its {\it constant} and {\it linear} parts, $\kappa^C:I \rightarrow H$ and $\kappa^L:I \rightarrow V \otimes H$, respectively. 
\end{itemize}
\end{hypothesis}

\begin{notation} [$\mc{D}_{R,\kappa}$]   \label{def:D Akap}
Let $\mc{D}_{R,\kappa}$ be the filtered $k$-algebra given by
$$\mc{D}_{R,\kappa} = \frac{T(V) \# H}{\left(r - \kappa(r)\right)_{r\in I}}.$$
Here, we assign the elements of $H$ degree 0. 
\end{notation}
\smallskip

The main result of \cite{WW} is the following.
The action of $H$ on itself that is used in the theorem below is the left adjoint action, that is, $h\cdot \ell = \sum h_1 \ell S(h_2)$ for all $h,\ell\in H$. 

\begin{theorem}\cite[Theorem~3.1]{WW} \label{thm:mainWW} Let $H$ be a Hopf algebra with bijective antipode. 
Then, the algebra $\mc{D}_{R,\kappa}$ is a PBW deformation of $R\# H$ if and only if the following conditions hold:
\begin{enumerate}
\item $\kappa$ is $H$-invariant, i.e.\ $\kappa(h \cdot r) = h \cdot \kappa(r)$ for all $r \in I$, 
\item[(b)] ${\rm Im}(\kappa^L \otimes \id - \id \otimes \kappa^L) \subseteq I$,

\item[(c)] 
$\kappa^L \circ (\kappa^L \otimes \id - \id \otimes \kappa^L) = -(\kappa^C \otimes \id - \id \otimes \kappa^C)$, and 

\item[(d)] 
$\kappa^C \circ ( \id \otimes \kappa^L - \kappa^L \otimes \id) \equiv 0,$
\end{enumerate}
where the maps $\kappa^C \otimes \id - \id \otimes \kappa^C$ and $\kappa^L \otimes \id - \id \otimes \kappa^L$ are defined on the intersection $(I \otimes V) \cap (V \otimes I)$.

\noindent Moreover, if  $\kappa^L \equiv 0$, then $\mc{D}_{R,\kappa}$ is a PBW deformation of $R\# H$ if and only if {\em (a)} above holds,  and
\begin{enumerate}
\item[(c$'$)] 
$\kappa \otimes \id = \id \otimes \kappa,$ on $(I \otimes V) \cap (V \otimes I)$. \qed
\end{enumerate}
\end{theorem}

In this paper we will focus on deformations for which $\kappa^L\equiv 0$, and thus will work
with conditions (a) and~(c$'$). We will show that there are many interesting new such examples, and we speculate about the more general setting (see Problem~\ref{kappaL}).

The lemma below will be of use in computing PBW deformations.
Let $H^H$ denote the subalgebra of $H$ consisting of elements invariant under the left adjoint action, that is, $\ell\in H^H$
if and only if $h\cdot \ell = \epsilon(h)\ell$ for all $h\in H$. A straightforward calculation
shows that $H^H = Z(H)$, the center of $H$ as an algebra. 

\begin{lemma} \label{lem:prelim}
Suppose that $\kappa^L\equiv 0$ and   $r$ is an $H$-invariant element of $I$, that is, $h \cdot r = \epsilon(h) r$ for all $h \in H$. Then, the condition that $h \cdot \kappa(r) = \kappa(h \cdot r)$ for all $h \in H$ is equivalent to $\kappa(r) \in H^H= Z(H)$.
\end{lemma}

\begin{proof}
Note that $\kappa(h \cdot r) =  \kappa(\epsilon(h) r) = \epsilon(h) \kappa(r)$. So, $\kappa(r) \in H^H$ if and only if $h \cdot \kappa(r) = \kappa(h \cdot r)$ for all $h \in H$. 
 \end{proof}

We make the following standing assumption for the rest of the article.

\begin{hypothesis} \label{deg0}
Unless stated otherwise, we assume that the deformation parameter $\kappa^L$ is equal to 0, so that $\kappa = \kappa^C$, that is we only consider PBW deformations of degree 0.
\end{hypothesis}

Now we have that the {\it rational Cherednik algebras} fit into the context of Theorem~\ref{thm:mainWW} as follows:

\begin{example}\label{ex:rCa} 
Let $H =\CC \Gamma$, where $\Gamma$
is a complex reflection group whose natural representation is $U$.
Denote by $U^*$ the vector space dual to $U$ with corresponding dual action of $\Gamma$. Let $V = U\oplus U^*$. A {\it rational Cherednik algebra} (or a {\it symplectic reflection algebra}) is an algebra of the form
$$
   {\mathcal D} = \frac{T(V)\# \CC \Gamma} {(uv-vu-\kappa(uv-vu)\mid u,v\in V)}
$$
for a particular type of function $\kappa$ taking values in $\CC \Gamma$. 
Conditions (a) and (c$'$) of Theorem~\ref{thm:mainWW} are satisfied, and thus $\mathcal D$ is a PBW deformation of $S(V)\# \CC \Gamma \cong (S(U)\ot S(U^*))\# \CC \Gamma$. 
We refer the reader to \cite{Dr} or \cite[Theorem~1.3 and Corollary~4.4]{EG} for details.  
\end{example}
 

\section{Example: $H = U_q(\mathfrak{sl}_2)$ and  $A=k_q[u,v]$} \label{sec:U kq}

Let $k$ be an algebraically closed field of characteristic 0. 
Let $q$ be a primitive $n$-th root of unity, $n\geq 3$. 
Let $$A= k_q[u,v] = T(V) /(uv-qvu),$$ 
where $V$ is the vector space with basis $u,v$. 
Note that $A$  is well-known to be Koszul.

Take $H = U_q(\mathfrak{sl}_2)$ as in \cite{Jantzen}. It is generated by grouplike elements $K, K^{-1}$, a $(K,1)$-skew primitive element $E$, and a $(1,K^{-1})$-skew primitive element $F$, that is
$$
  \Delta(K^{\pm 1})=K^{\pm 1}\ot K^{\pm 1}, \ \ \
   \Delta(E)=E\ot 1 + K\ot E, \ \ \ \Delta(F)=F\ot K^{-1}+1\ot F.
$$
The relations of $U_q(\mathfrak{sl}_2)$ are 
$$EF-FE=\frac{K - K^{-1}}{q-q^{-1}}, 
\quad \quad KEK^{-1} = q^2 E, \quad \quad KFK^{-1} = q^{-2}F, \quad \quad KK^{-1} = K^{-1}K = 1.$$ 
The antipode $S$ is given by 
$S(K^{\pm 1}) = K^{\mp 1}$, $S(E)  = -K^{-1}E$, and $S(F) = -FK$.
The action of $H$ on $A$ is given by:
\[
\begin{array}{lll}
E \cdot u = 0, ~\quad  \quad &F \cdot u =v, ~\quad \quad & K^{\pm 1} 
     \cdot u = q^{\pm 1}u,\\
E \cdot v = u, ~ \quad \quad &F \cdot v =0, ~\quad \quad & K^{\pm 1} 
     \cdot v = q^{\mp 1}v.
\end{array}
\]
Since $q^n=1$, 
$K^n$ acts on $A$ as the identity and it may be checked that 
$E^{n}, F^{n}$ act as zero on $A$.

We will compute PBW deformations of $(A \otb A^\opc) \# H$ and of $(A \otimes^{\tau} A) \# H$, for a braiding $\sf c$ of a category of $H$-modules and for a twisting map $\tau$, respectively.

\subsection{PBW deformations of $(A\otb A^\opc) \# U_q({\mathfrak{sl}}_2)$} 
\label{uqsl2braid}
Let $\mathcal C$ be a category of $H$-modules on which $E$ (or $F$) 
acts locally nilpotently, and includes $A$ as an object. For example, 
 take the category of  locally finite-dimensional $H$-modules, which is a braided monoidal category. A braiding  may be described explicitly as follows. Let $M,M'$ be 
$U_q({\mathfrak{sl}}_2)$-modules, each   a direct sum of eigenspaces of $K$
with eigenvalues given by powers of $q$.
If $m\in M$ and $m'\in M'$ with 
$K\cdot m = q^a m$ and $K\cdot m' = q^b m'$, let  
\begin{equation} \label{braiduqsl2}
{\sf c}_{M,M'} (m\ot m') = 
q^{-\frac{1}{2} ab } \left(
\sum_{i=0}^{n-1} q^{-i (i-1)/2} \frac{(q^{-1}-q)^i}{[i]_q ! } F^i\ot E^i 
\right) (m'\ot m) , 
\end{equation}
where $[i]_q ! = [i]_q [i-1]_q \cdots [1]_q$ and $[j]_q = (q^j-q^{-j})/ (q-q^{-1})$. 
For further details, see \cite{Rosso} or \cite[Chapter~3]{Jantzen}; the latter may be
modified in this root of unity case.  
The reader may also wish to compare with a similar formula in 
\cite[Theorem XVII.4.2]{Kassel}.

We first give details about the structure of $A^\opc$ and $A\otb A^\opc$
as $U_q({\mathfrak{sl}}_2)$-module algebras, then we consider 
PBW deformations of $(A\otb A^\opc ) \# U_q({\mathfrak{sl}}_2)$.

\begin{proposition}\label{prop:structure}
The structure of $A^\opc$, $A\otb A^\opc$, and the action of $U_q({\mathfrak{sl}}_2)$ 
 is given as follows. 
\begin{enumerate}
\item $A^\opc$ is isomorphic to $A$ as an
$H$-module algebra: $ \  A^\opc\cong k_q[u',v'] \cong T(V')/(u'v'-qv'u'),$
where $V'$ is a copy of the $H$-module $V$ with basis $u',v'$. 
\item The relations of $A \otb A^\opc$ are
\[
\begin{array}{lll}
r_1 := uv - qvu, && r_2 := u'v' - qv'u' ,\\
r_3 := uu'-q^{-\frac{1}{2}}u'u, && r_4 := vu'-q^{\frac{1}{2}}u'v - (q^{-\frac{1}{2}} - q^{\frac{3}{2}}) v'u,\\
r_5 := uv' -  q^{\frac{1}{2}}v'u, && r_6 := vv' - q^{-\frac{1}{2}}v'v .
\end{array}
\]
\item The action of $H$ on the relations of $A\otb A^\opc$ induced by its action
on $T(V\oplus V')$ is 
\[\begin{array}{lll}
K\cdot r_1 = r_1 , & E\cdot r_1 = 0 , & F\cdot r_1 = 0 ,\\
K\cdot r_2 = r_2, & E\cdot r_2 = 0, & F\cdot r_2 = 0 ,\\
K\cdot r_3 = q^2 r_3 , & E\cdot r_3 = 0 , & F\cdot r_3 = r_5 + q^{-1} r_4,\\
K\cdot r_4 = r_4 , & E\cdot r_4 = r_3 , & F\cdot r_4 = r_6,\\
K\cdot r_5 = r_5 , & E\cdot r_5 = qr_3 , & F\cdot r_5 = qr_6 ,\\
K\cdot r_6 = q^{-2} r_6 , & E\cdot r_6 = r_5 + q^{-1} r_4, & F\cdot r_6 = 0.\\
\end{array}\]
\end{enumerate}
\end{proposition}

\begin{proof}
(a) Let us compute the braided-opposite algebra $A^\opc$, generated by copies $u'$ and $v'$ of $u$ and $v$, respectively. Since $K\cdot u = qu$ and $K\cdot v = q^{-1}v$, using the braiding~\eqref{braiduqsl2}, we have 
\[
\begin{array}{ll}
u'u' = m_A (q^{-\frac{1}{2}} (1\ot 1 + (q^{-1}-q) F\ot E) (u\ot u)) &= q^{-\frac{1}{2}}u^2,\\
u'v' = m_A (q^{\frac{1}{2}} (1\ot 1 + (q^{-1}-q) F\ot E) (v\ot u)) &= q^{\frac{1}{2}}vu,\\
v'u' = m_A (q^{\frac{1}{2}} (1\ot 1 + (q^{-1}-q) F\ot E) (u\ot v)) &= q^{\frac{1}{2}}uv+(q^{-\frac{1}{2}}-q^{\frac{3}{2}})vu,\\
v'v' = m_A (q^{-\frac{1}{2}} (1\ot 1 + (q^{-1}-q) F\ot E) (v\ot v)) &= q^{-\frac{1}{2}} v^2.\\
\end{array}
\]
So, $u'v' = q^{\frac{1}{2}}vu =  q^{\frac{1}{2}}vu  + q^{\frac{3}{2}}(uv-qvu) =  q^{\frac{3}{2}}uv+(q^{\frac{1}{2}}-q^{\frac{5}{2}})vu = qv'u'$. Now it is clear that $A^\opc$ is generated by $u', v'$ subject to relation $r_2$.

(b) This is verified in the same manner as part (a); for example, we establish relations $r_4$ and $r_5$ as follows: 
\[
\begin{array}{l}
v u' = m_A (q^{\frac{1}{2}} (1\ot 1 + (q^{-1}-q) F\ot E) (u'\ot v)) 
= q^{\frac{1}{2}}u'v+(q^{-\frac{1}{2}} - q^{\frac{3}{2}})v'u,\\
u v' = m_A (q^{\frac{1}{2}} (1\ot 1 + (q^{-1}-q) F\ot E) (v'\ot u)) 
= q^{\frac{1}{2}}v'u.
\end{array}
\]

(c) The verification is straightforward. For instance,
\[
\begin{array}{ll}
E\cdot r_6 = E \cdot (vv'-q^{-\frac{1}{2}}v'v) &= 
(K \cdot v)(E \cdot v') + (E \cdot v)v' 
- q^{-\frac{1}{2}}((K \cdot v')(E \cdot v) + (E \cdot v')v)\\
& 
= q^{-1} vu' + uv' 
- q^{-\frac{1}{2}}(q^{-1} v'u + u'v)\\
&= r_5 + q^{-1} r_4.
\end{array}
\]

\vspace{-.23in}

\end{proof}

\begin{remark} \label{rk:RCA-dual}
The braided opposite algebra $A^\opc$ is isomorphic to  $$B:=k_{q^{-1}}[u^*,v^*]:= k\langle u^*, v^* \rangle/(u^* v^* - q^{-1} v^* u^*),$$ the algebra generated by functions $u^*,v^*$ in $V^*$ dual to $u,v$ in $V$. The action of $U_q({\mathfrak{sl}}_2)$ on $B$ is given via the antipode $S$. 
One checks that 
\[
\begin{array}{lll}
E\cdot u^* = -q^{-1}v^*, ~\quad \quad &F \cdot u^*=0, ~\quad \quad & K^{\pm 1}
   \cdot u^* = q^{\mp 1} u^*,\\
E\cdot v^*=0, ~\quad \quad &F\cdot v^* = -qu^*, ~\quad \quad & K^{\pm 1}
   \cdot v^* = q^{\pm 1} v^*.
\end{array}
\]
For instance, 
\[
\begin{array}{llll}
E \cdot u^*(u) &= K(u^*(S(E) \cdot u)) + E(u^*(S(1) \cdot u)) 
&= K(u^*(-K^{-1}E \cdot u)) + E(u^*(1 \cdot u))  &= 0\\
E \cdot u^*(v) &= K(u^*(S(E) \cdot v)) + E(u^*(S(1) \cdot v)) 
&= K(u^*(-K^{-1}E \cdot v)) + E(u^*(1 \cdot v))  &=  -q^{-1}. 
\end{array}
\]
The function $\phi: A^\opc\rightarrow B$ defined by 
$
    \phi(u') = -q^{-1}v^*$  and $ \phi(v') = u^*
$
is an $H$-module algebra isomorphism. 
Thus these examples are more similar to the rational Cherednik algebras than may at first appear; compare to Example~\ref{ex:rCa}.
\end{remark}

We next give some examples of PBW deformations of $(A \otb A^\opc) \#U_q({\mathfrak{sl}}_2)$ when $q$ is a primitive {\em third} root of unity.
For other values of $q$, the possible deformations are limited, and we will instead turn to a more general twisted tensor product in the next subsection before presenting nontrivial deformations.

\begin{proposition}\label{prop:third-root}
Let $q$  be a primitive third root of unity and $A=k_q[u,v]$. 
Consider the algebra $A\otb A^\opc$  generated by $u,v,u',v'$, subject to the relations $r_1,\ldots,r_6$
as in Proposition~\ref{prop:structure}(b). 
Then, the smash product $(A \otb A^\opc) \# U_q(\mathfrak{sl}_2)$ admits PBW deformations 
\[
    {\mathcal D}_{A\otb A^\opc, \kappa} = 
   \frac{k\langle u,v,u',v'\rangle \# U_q({\mathfrak{sl}}_2)} {(r_i-\kappa(r_i))}
\]
where 
\[
\begin{array}{ll}
\kappa(r_1) =
K^{3s_1}, E^{3t_1}, \text{ or } F^{3t_2},  
\quad&\quad 
\kappa(r_2) =K^{3s_2}, E^{3t_3}, \text{ or } F^{3t_4},  \\
 \kappa(r_3) = 0 , \quad&\quad
\kappa(r_4) = \alpha  K^{3s_0}, \\
 \kappa(r_5) =  -\alpha q^2 K^{3s_0}, \quad&\quad \kappa(r_6) = 0,
\end{array}
\]
for any $s_i \in \mathbb{Z}$, $t_i \in \mathbb{N}$, and $\alpha \in k$.
\end{proposition}

\begin{proof}
In equation (\ref{braiduqsl2}), we take $q^{\frac{1}{2}}$ to be $q^2$, the
other primitive third root of unity. 
We verify conditions (a) and (c$'$) of Theorem~\ref{thm:mainWW} where $V$
is replaced by $V\oplus V'$ with basis $u,v,u',v'$. 
Now for (a), by Proposition~\ref{prop:structure}(c), as $q^{-1}=q^2$, we need to show that 
\[\begin{array}{lll}
K\cdot \kappa(r_1) = \kappa(r_1) , & E\cdot \kappa(r_1) = 0 , & F\cdot \kappa(r_1) = 0 ,\\
K\cdot \kappa(r_2) = \kappa(r_2), & E\cdot \kappa(r_2) = 0, & F\cdot \kappa(r_2) = 0 ,\\
K\cdot \kappa(r_3) = q^2 \kappa(r_3) , & E\cdot \kappa(r_3) = 0 , & F\cdot \kappa(r_3) = \kappa(r_5) + q^2 \kappa(r_4),\\
K\cdot \kappa(r_4) = \kappa(r_4) , & E\cdot \kappa(r_4) = \kappa(r_3), & F\cdot \kappa(r_4) = \kappa(r_6),\\
K\cdot \kappa(r_5) = \kappa(r_5) , & E\cdot \kappa(r_5) = q\kappa(r_3) , & F\cdot \kappa(r_5) = q \kappa(r_6),\\
K\cdot \kappa(r_6) = q \kappa(r_6) , & E\cdot \kappa(r_6) = \kappa(r_5) + q^2 \kappa(r_4), & F\cdot \kappa(r_6) = 0.\\
\end{array}\]
By noting that $K \cdot \ell = K \ell K^{-1}$, $E \cdot \ell = -K\ell K^{-1}E + E\ell$, and $F \cdot \ell = -\ell FK + F\ell K$, for $\ell \in H$, the equations above are easy to check. For instance, say $\kappa(r_2) = E^{3t}$, then
\[
K \cdot \kappa(r_2) = K E^{3t} K^{-1} = q^{6t} E^{3t} = E^{3t} = \kappa(r_2)
 \quad  \text{and}  \quad 
F \cdot \kappa(r_3) = 0 = -\alpha q^2 K^{3s_0} +  q^2 \alpha K^{3s_0} = \kappa(r_5) + q^2 \kappa(r_4).
\]

Now we verify Theorem~\ref{thm:mainWW}(c$'$), that $\kappa \otimes \id = \id \otimes \kappa$ as maps on $(I \otimes W) \cap (W \otimes I)$, where $W=V\oplus V'$. Since $\dim W = 4$, $\dim I =6$, and $\dim (A \otb A^\opc)_3 = {3 + (4-1) \choose 4-1}= 20$ (considering the PBW basis of $A \otb A^\opc$), we get 
$$\dim ((I \otimes W) \cap (W \otimes I)) 
~=~  \dim(I \otimes W) + \dim(W \otimes I)  - \dim W^{\otimes 3} + \dim(A \otb A^\opc)_3 ~=~ 4.$$
The following equations show that the indicated elements are in $(I\otimes W)
\cap (W\otimes I)$.  They are visibly linearly independent, and therefore form 
a  basis of $(I \otimes W) \cap (W \otimes I)$ as a subspace of $T(W)$:
 \[
\begin{array}{rl}
r_1 u' - q^{{2}} r_3 v + q^{{2}}r_4 u - (q- 1)r_5 u 
&= u'r_1- q v r_3 +ur_4,\\
r_1 v' - q r_5 v +  r_6 u &= v' r_1 - q v r_5 +u r_6,\\
r_2 u + r_3 v' -q r_5 u' &= u r_2 + v' r_3 - qu'r_5,\\
r_2 v+ r_4 v' -q r_6 u' &= vr_2 + q^{{2}} v' r_4 - (q - 1) v' r_5 - q^{{2}}u'r_6.
\end{array}
\]
Theorem~\ref{thm:mainWW}(c$'$) is thus equivalent to the conditions: 
 \[
\begin{array}{c}
\kappa(r_1) u' - q^{{2}} \kappa(r_3) v + q^{{2}}\kappa(r_4) u - (q- 1)\kappa(r_5) u 
\quad=\quad u'\kappa(r_1)- q v \kappa(r_3) +u\kappa(r_4),\\
\kappa(r_1) v' - q \kappa(r_5) v +  \kappa(r_6) u \quad=\quad v' \kappa(r_1) - q v \kappa(r_5) +u \kappa(r_6),\\
\kappa(r_2) u + \kappa(r_3) v' -q \kappa(r_5) u' \quad=\quad u \kappa(r_2) + v' \kappa  (r_3) - q u'\kappa(r_5,)\\
\kappa(r_2) v+ \kappa(r_4) v' - q \kappa(r_6) u' \quad=\quad v\kappa(r_2) + q^{{2}} v' \kappa(r_4) - (q - 1) v' \kappa(r_5) - q^{{2}}u'\kappa(r_6).
\end{array}
\]
Note that $E^3$ is $(K^3, 1)$-skew primitive and $F^3$ is $(1,K^{-3})$-skew primitive since $q$ is a third root of unity. 
Thus in $(A\otb A^\opc )\# U_q({\mathfrak{sl}}_2)$, for example, 
$E^3 u' = (K^3\cdot u') E^3 + (E^3\cdot u') = u' E^3$. 
By direct computation,  we see that each of $\kappa(r_1),\kappa(r_2)$ commutes with each of $u,v,u',v'$; hence it suffices to show that 
 \[
\begin{array}{c}
 - q^{{2}} \kappa(r_3) v + q^{{2}}\kappa(r_4) u - (q- 1)\kappa(r_5) u 
\quad=\quad - q v \kappa(r_3) +u\kappa(r_4),\\
 - q \kappa(r_5) v +  \kappa(r_6) u \quad=\quad  - q v \kappa(r_5) +u \kappa(r_6),\\
 \kappa(r_3) v' -q \kappa(r_5) u' \quad=\quad  v' \kappa(r_3) - qu'\kappa(r_5),\\
 
 \kappa(r_4) v' - q \kappa(r_6) u' \quad=\quad q^{{2}} v' \kappa(r_4) - (q - 1) v' \kappa(r_5) - q^{{2}}u'\kappa(r_6).
\end{array}
\]
These equations are satisfied by the given values of
$\kappa(r_3), \kappa(r_4), \kappa(r_5), \kappa(r_6)$.
\end{proof}

 \subsection{PBW deformations of $(A \otimes^\tau A) \# U_q(\mathfrak{sl}_2)$}\label{subsec:tau}
 
Now we let $q$ be a primitive $n$-th root of unity, $n\geq 3$.
In this section we generalize the previous example in the context of
twisted tensor products.
Let $A = k_q[u,v]$ and $A'=k_q[u',v']$
with the actions of $U_q(\mathfrak{sl}_2)$ as given before. 
Define $\tau: A\ot A'\rightarrow A'\ot A$  inductively on the basis
$u^iv^j\ot (u')^r(v')^s$ by first defining
\begin{equation} \label{twistuqsl2}
\begin{array}{lll}
\tau(u\ot u')  =  q u'\ot u,&& \tau(u\ot v') =   q^2 v'\ot u, \\
  \tau(v\ot u') =  q^2 u'\ot v + (q-q^3)v' \ot u, && \tau(v\ot v')  =  qv'\ot v.
  \end{array}
\end{equation} 
In the case that $n=3$, we see that $\tau$ is the same as the braiding given in
the previous section.
For other values of $n$, one may apply \cite[Theorem~3.4.7]{BGV} as follows to show 
 that $\tau$ is well-defined and 
satisfies the associativity constraint.
First define an algebra $R$ to be generated by $u,v,u',v'$ with relations 
\[
\begin{array}{lll}
r_1 :=  uv - qvu , &&
r_2 :=  u'v' - qv'u',\\
r_3 := uu'- q u'u, &&
r_4 :=  vu' - q^2 u'v - (q - q^3) v'u,\\
r_5 := uv' - q^2 v'u, &&
r_6 := vv' - q v'v .
\end{array}
\] 
We will prove that $R$ has basis $\{(u')^{i_1}(v')^{i_2}u^{i_3} v^{i_4} \mid i_1,i_2,i_3,i_4\in \N \}$,
so that as a vector space, it may be identified with $A'\ot A$. 
As it is also an associative algebra, this proves that $\tau$ 
satisfies the associativity constraint and $R\cong A'\ot^{\tau} A$. 
We order the generators $u',v',u,v$, and correspondingly impose the degree-reverse-lex order on $\N^4$.
The relations $r_1,\ldots, r_6$ above give rise to a bounded quantum reduction system
(see \cite[Definition 4.1]{BGV}) by rewriting, using the relations, to define $f_{vu}$, $f_{v'u'}$, $f_{uu'}$, $f_{vu'}$, $f_{uv'}$, $f_{vv'}$ as follows:
\[
\begin{array}{lll}
  vu  =  f_{vu} :=  q^{-1}uv, &&
  v'u' = f_{v'u'} := q^{-1}u'v' \\
   uu' = f_{uu'} :=  qu'u, &&
  vu' = f_{vu'} :=  q^2 u'v + (q-q^3) v'u,\\
  uv' = f_{uv'} :=  q^2 v'u, &&
   vv' = f_{vv'} := qv'v.
\end{array}
\]
By \cite[Theorem 3.4.7]{BGV}, $R$ has the basis claimed if and only if 
$z f_{yx} = f_{zy} x$ for all ordered triples $(x,y,z)$ of generators. 
There are 4 generators, and thus 4 such triples. 
All triples satisfy the required condition; as one example,
\[
\begin{array}{lll}
   v f_{v'u'} = v (q^{-1}u'v')  &=  q^{-1} (q^2 u'v + (q-q^3) v'u) v' 
   &= q u' v v' + (1-q^2) v' u v' ~= q^2 u' v' v + (q^2-q^4) (v')^2 u, \\
  f_{vv'} u' = (qv'v) u' &=  qv' (q^2 u'v + (q-q^3) v'u)
   &= q^2 u'v'v +  (q^2-q^4) (v')^2 u , 
\end{array}
\]
and these two expressions are indeed equal. 
Thus we have a twisted tensor product algebra $A'\ot^{\tau} A$ that is isomorphic to
the algebra with generators $u,v,u',v'$ and relations $r_1,\ldots, r_6$ given above. 
A calculation shows that the twisting map $\tau$ is an $H$-module homomorphism.
Therefore $A'\ot^{\tau} A$ is an $H$-module algebra. 
For applying Theorem~\ref{thm:mainWW}, we will need to know 
that in $T(V\oplus V')$, 
$$
\begin{array}{lll}
K\cdot r_1 = r_1 , & E\cdot r_1 = 0 , & F\cdot r_1 = 0 ,\\
K\cdot r_2 = r_2, & E\cdot r_2 = 0, & F\cdot r_2 = 0 ,\\
K\cdot r_3 = q^2 r_3 , & E\cdot r_3 = 0 , & F\cdot r_3 = r_5 + q^{-1} r_4 ,\\
K\cdot r_4 = r_4 , & E\cdot r_4 = r_3 , & F\cdot r_4 = r_6 ,\\
K\cdot r_5 = r_5 , & E\cdot r_5 = qr_3 , & F\cdot r_5 = qr_6 ,\\
K\cdot r_6 = q^{-2} r_6 , & E\cdot r_6 = r_5 + q^{-1} r_4 , & F\cdot r_6 = 0 .
\end{array}
$$
We will also need to know  a basis of $(I\ot W)\cap (W\ot I)$, where $W=V\oplus V'$. One checks that  
\[
\begin{array}{rl}
r_1 u' - q^2 r_3 v + q^2 r_4 u - (q-q^3) r_5 u &= 
  u'r_1 - qvr_3 + ur_4,\\
r_1 v' - q r_5 v + q^3 r_6 u &=  q^3 v' r_1 - q vr_5 + u r_6,\\
q^3 r_2 u + r_3 v' - q r_5 u' &= ur_2 + q^3 v'r_3 - qu'r_5,\\
q^3 r_2 v + r_4 v' -qr_6 u' &= v r_2 + q^2 v'r_4 - (q-q^3)v' r_5-q^2 u' r_6.
\end{array}
\]
These elements form a basis of $(I\ot W)\cap (W\ot I)$. 
Now the conditions in Theorem~\ref{thm:mainWW} may be checked, as in the proof
of Proposition~\ref{prop:third-root}, to show  the following. 

\begin{proposition}\label{prop:qn} 
Let $q$ be a primitive $n$-th root of unity, $A=k_q[u,v]$, $A'=k_q[u',v']$.
The algebra $A'\ot ^{\tau} A$ is generated by $u,v,u',v'$, subject to the
relations $r_1,\ldots, r_6$ above. 
The smash product $(A'\ot^{\tau} A) \# U_q({\mathfrak{sl}}_2)$ admits PBW
deformations 
\[ {\mathcal D}_{A'\ot^{\tau} A, \kappa} = 
  \frac{k\langle u,v,u',v'\rangle \# U_q({\mathfrak{sl}}_2)}
     {(r_i-\kappa(r_i))}
\]
with 
$$
\kappa(r_4) = \alpha K^{ns}, \ \ \ \ \ \kappa(r_5) = -q^{-1} \alpha K^{ns},
$$
where $\alpha$ is an arbitrary scalar, $s$ an arbitrary integer, 
and $\kappa$ applied to each of the other relations $r_i$ is 0. 
\qed
\end{proposition} 

\begin{remark} \label{rk:uqgl2}
The examples in this section may be extended to $U_q({\mathfrak{gl}}_2)$ to
obtain PBW deformations of $(A'\ot^{\tau} A)\# U_q({\mathfrak{gl}}_2)$.
Let $q \in k$ be a root of unity of  order $n \geq 3$.
The  Hopf algebra $U_q(\mathfrak{gl}_2)$ is
 generated as an algebra by $G_1^{\pm 1}$, $G_2^{\pm 1}$, $E$, and $F$, subject to relations:
\[
\begin{array}{c}
\medskip

$$G_1E = q E G_1, \quad G_2 E = q^{-1}EG_2, \quad G_1 F = q^{-1}F G_1, \quad G_2 F = q F G_2,\\
G_1 G_2 = G_2 G_1, \quad G_1G_1^{-1} = G_1^{-1}G_1 = 1, 
\quad G_2G_2^{-1}=G_2^{-1}G_2 =1, 
\quad EF - FE = \displaystyle \frac{G_1G_2^{-1} - G_2G_1^{-1}}{q - q^{-1}}.
\end{array}
\]
The Hopf structure of $U_q(\mathfrak{gl}_2)$ and its action on $A=k_q[u,v]$ are given by:
\[
\begin{array}{lll}
\Delta(E) = E \ot 1 + G_1G_2^{-1} \ot E, \quad &\Delta(F) = F \ot G_2G_1^{-1} + 1 \ot F, \quad &\Delta(G_i^{\pm 1}) = G_i^{\pm 1} \ot G_i^{\pm 1},\\
\epsilon(E) = 0, ~&\epsilon(F) = 0 ~ &\epsilon(G_i^{\pm 1}) = 1,\\
S(E) = -G_2G_1^{-1}E, ~ &S(F) = -FG_1G_2^{-1}, ~ &S(G_i^{\pm 1}) = G_i^{\mp 1},
\end{array}
\]
for $i = 1,2$ and 
\vspace{-.1in}

\[
\begin{array}{llll}
E \cdot u = 0, ~\quad  \quad &F \cdot u =v, ~\quad \quad & G_1 \cdot u = qu, ~\quad  \quad & G_2 \cdot u = u,\\
E \cdot v = u, ~ \quad \quad &F \cdot v =0, ~\quad \quad & G_1 \cdot v = v, ~\quad  \quad & G_2 \cdot v = qv.
\end{array}
\]
Note that the twisted tensor product $A'\ot^{\tau} A$  is preserved
by the actions of $G_1$ and $G_2$, and setting $K = G_1 G_2^{-1}$ allows us to
realize $A'\ot^{\tau} A$ as a $U_q(\mathfrak{gl}_2)$-module algebra.
The rest of the development in the case $U_q(\mathfrak{sl}_2)$ works here as well,
resulting in PBW deformations 
\[
   {\mathcal D}_{A'\ot^{\tau} A, \kappa} = 
  \frac{k\langle u,v,u',v'\rangle \# U_q({\mathfrak{gl}}_2)}
    {(r_i-\kappa (r_i))}
\]
of $(A'\ot^{\tau} A)\# U_q(\mathfrak{gl}_2)$ with
$$
  \kappa (r_4) = \alpha G_1^{ns} G_2^{-ns} , \ \ \ \kappa(r_5) = -q^{-1} \alpha G_1^{ns}G_2^{-ns},
$$
and $\kappa$ applied to each of the other relations is 0. 
(Cf.\ Proposition~\ref{prop:qn}.) 
\end{remark}


\section{Example: $H = T(2)$ (Sweedler Hopf algebra) and $A=k[u,v]$} \label{sec:k[u,v]}

Take a field $k$ with  $\text{char}(k) \neq 2$. Consider the Sweedler (Hopf) algebra $T(2)$ generated by a grouplike element $g$ and a $(g,1)$-skew primitive element $x$ with relations:
$$g^2 = 1, \quad \quad x^2=0, \quad \quad gx+xg=0.$$
Here, $\epsilon(g) =1$, $\epsilon(x) = 0$, $S(g) = g$, and $S(x) = -gx$. Further, by \cite[page 296]{Ra}, $T(2)$ is quasitriangular with $\text{R}$-matrix depending on a parameter $\lambda\in k$: 
\begin{equation}\label{T2Rmatrix}
\text{R}_{\lambda} = \textstyle \frac{1}{2}(1 \otimes 1 + 1 \otimes g + g \otimes 1 - g \otimes g) +
\frac{\lambda}{2}(x \otimes x + x \otimes gx + gx \otimes gx - gx \otimes x) .
\end{equation}
The corresponding braiding ${\sf c}={\sf c}_\lambda$ is given by ${\sf c}_\lambda = \text{R}_\lambda \circ\sigma$
where $\sigma$ is the flip map: $\sigma(m\ot m')=m'\ot m$ for all $m\in M$, $m'\in M'$,
where $M,M'$ are two $T(2)$-modules. 
There is a $T(2)$-action on $A= k[u,v]$ given by
$$g \cdot u = u, \quad g \cdot v = -v, \quad x \cdot u =0, \quad x \cdot v =u,$$
under which $A$ is a $T(2)$-module algebra. 
Let $V$ be the vector space with basis $u,v$ and 
write $A=T(V)/(r)$, with $r=uv-vu$. The action of $H$ on $T(V)$ yields 
$$g \cdot r = -r \quad \quad \text{and} \quad \quad x \cdot r = 0.$$

Let us compute the $T(2)$-opposite algebra $A^\opc$, generated by $u'$ and $v'$. Here,
\[
\begin{array}{ll}
u'u' = m_A  \text{R}_{\lambda}(u \otimes u)  = u^2,  &\quad \quad u'v' = m_A  \text{R}_{\lambda}(v \otimes u)  = vu,\\
v'u' =  m_A  \text{R}_{\lambda}(u \otimes v) = uv, &\quad \quad v'v' = m_A  \text{R}_{\lambda}(v \otimes v)  = -v^2+\lambda u^2.
\end{array}
\]
So, the only relation of $A^\opc$ is $u'v' = v'u'$, and  $A^\opc = k[u',v']$. 
(Note that there is an isomorphism $A^\opc \cong k[u^*,v^*]$, where $u^*$, $v^*$ are dual basis vectors of $V^*$, with  
$\ g\cdot u^*=u^*, \ g\cdot v^*=-v^*, \ x\cdot u^* = -v^*, \ x\cdot v^*=0$. 
For example,
\[ 
\begin{array}{l}
x \cdot u^*(u) = g(u^*(S(x)\cdot u))+x(u^*(S(1)\cdot u)) = g(u^*(0))+x(u^*(u)) =0\\
x \cdot u^*(v) = g(u^*(S(x)\cdot v))+x(u^*(S(1)\cdot v)) = g(u^*(-u))+x(u^*(v)) =-1,
\end{array}
\]
which implies $x \cdot u^* = -v^*$.)

The braided product  $A \otb  A^\opc$
is generated by $u,v,u',v'$, with relations 
\[
\begin{array}{lll}
r_1 := uv-vu, &&
r_2 := u'v'-v'u',\\
r_3 := uv'-v'u, &&
r_4 := vv'+v'v-\lambda uu',\\
r_5 := uu'-u'u,&&
r_6 := vu'-u'v.
\end{array}
\]
To obtain $r_4$, for instance, consider the following calculation:
\[
\begin{array}{ll}
(m_A \ot m_{A^\opc})(1 \ot {\sf c} \ot 1)(1 \ot v' \ot v \ot 1)
&= \textstyle \frac{1}{2}(vv'-vv'-vv'-vv') + \frac{\lambda}{2}(uu'+uu'+uu'-uu')\\
&= -vv' + \lambda uu'.
\end{array}
\]
Let $V'$ be the vector space with basis $u',v'$.
The $T(2)$-action on the relations considered as elements of $(V\oplus V')\ot
(V\oplus V')$ is given by 
\[
\begin{array}{lllllllllll}
g \cdot r_1 = -r_1, && g \cdot r_2 = -r_2, && g \cdot r_3 = -r_3, && 
g \cdot r_4 = r_4, && g \cdot r_5 = r_5,  && g \cdot r_6 = -r_6,\\
x \cdot r_1 = 0 ,    && x \cdot r_2 = 0,    && x \cdot r_3 = r_5, && 
x \cdot r_4 = r_3-r_6,  && x \cdot r_5 = 0,  && x \cdot r_6 = r_5.\\
\end{array}
\]

Now using Theorem \ref{thm:mainWW}, we may find PBW deformations $\mathcal{D}_{A \otb A^\opc, \kappa}$ of $(A \otb A^\opc) \# T(2)$. Let $A \otb A^\opc$  be presented as $T( V\oplus V' )/ (I)$, where $I$ is the subspace of $(V\oplus V')\ot (V\oplus V')$ spanned by  $r_1,r_2,r_3,r_4,r_5,r_6$.

\begin{proposition} \label{prop:T(2)}
Let $A=k[u,v]$. The algebra $A\otb A^\opc$ is generated by $u,v,u',v'$ subject
to the relations $r_1,\ldots, r_6$ above. 
There is a  1-parameter deformation of $(A \otb A^\opc) \# T(2)$, 
$${\mathcal {D}}_{A\otb A^\opc, \kappa} =
  \frac{k\langle u,v,u',v'\rangle \# T(2) } {(r_i-\kappa(r_i))}
$$
given by
 $\kappa(r_4) = \alpha\in k$ ~~and $\kappa(r_i) = 0$
for $i\neq 4$. 
Moreover, all PBW deformations for which $\kappa = \kappa^C$ (as assumed in Hypothesis~\ref{deg0}) are of this form. 
\end{proposition}

\begin{proof}
We have that $A \otb A^\opc$ is Koszul by Corollary~\ref{cor:twistkoszul} and Proposition~\ref{lem:HopKoszul}, and so by Theorem~\ref{thm:mainWW}, 
$\mathcal{D}_{A \otb A^\opc, \kappa}$  
is a PBW deformation of $(A \otb A^\opc)\# T(2)$ if and only if 
(a)$~~\kappa$ is $T(2)$-invariant and 
(c$'$)$~~ \kappa \otimes \id = \id \otimes \kappa$ as maps defined on the intersection $(I \otimes W) \cap (W \otimes I)$, where $W=V\oplus V'$. 

For each $i$, let $\alpha_i,\beta_i, \gamma_i,\delta_i\in k$ and consider
the function $\kappa$ determined by 
$$\kappa(r_i) = \alpha_i + \beta_i g + \gamma_i x + \delta_i gx, \text{ for } i =1, \dots, 6.$$
Applying the generators $g$ and $x$ to these elements, we have 
\[
\begin{array}{l}
\smallskip
g \cdot (\alpha_i + \beta_i g + \gamma_i x + \delta_i gx) = 
\alpha_i  + \beta_i g + \gamma_i gxg^{-1} + \delta_i xg^{-1} = \alpha_i + \beta_i g - \gamma_i x - \delta_i gx,\\
x \cdot (\alpha_i + \beta_i g + \gamma_i x + \delta_i gx)\\
\quad = 
\alpha_i gS(x) + \alpha_i xS(1) + \beta_i g^2 S(x) + \beta_i xgS(1) 
+ \gamma_i gxS(x) + \gamma_i x^2S(1) + \delta_i g^2xS(x) + \delta_i xgxS(1)\\
\quad = -\alpha_i x + \alpha_i x - \beta_i gx - \beta_i gx \\
\quad = -2 \beta_i gx.
\end{array}
\]
Condition (a) of Theorem~\ref{thm:mainWW} is equivalent to the following two sets of conditions: From the equation $g \cdot \kappa(r_i) = \kappa(g \cdot r_i)$ for $i=1, \dots, 6$, we get
$$\alpha_1 = \beta_1 = 0, \quad \alpha_2 = \beta_2 = 0, \quad\alpha_3 = \beta_3 = 0,\quad
\gamma_4 = \delta_4= 0, \quad \gamma_5 = \delta_5= 0, \quad \alpha_6 = \beta_6 = 0.$$
From the equation $x \cdot \kappa(r_i) = \kappa(x \cdot r_i)$ for $i=1, \dots, 6$, we get
\[
\beta_1= \beta_2= \alpha_5=\beta_5=\gamma_5=0,  \quad \delta_5=-2\beta_3,  \quad
\alpha_3 - \alpha_6 = \beta_3-\beta_6 = \gamma_3-\gamma_6 =0, \quad
\delta_3-\delta_6=-2\beta_4,  \quad \delta_5 = -2\beta_6.
\]
Putting this together,  condition (a) of Theorem~\ref{thm:mainWW} holds if and
only if:
\begin{eqnarray} \label{eq:T(2)kappa}
\begin{array}{lll}
\kappa(r_1) = \gamma_1 x + \delta_1 gx,&&
\kappa(r_2) = \gamma_2 x + \delta_2 gx,\\
\kappa(r_3) = \gamma_3 x + \delta_3 gx, &&
\kappa(r_4) = \alpha_4 + \beta_4 g,\\
\kappa(r_5) = 0, &&
\kappa(r_6) =  \gamma_3 x + (\delta_3+2\beta_4) gx,
\end{array}
\end{eqnarray}
for arbitrary scalars $\alpha_4,\beta_4,\gamma_1,\gamma_2,\gamma_3,\delta_1,\delta_2,\delta_3$.

Let us compute $(I \otimes W) \cap (W \otimes I)$, where $W=V\oplus V'$, and use the condition $\kappa \otimes \id = \id \otimes \kappa$ to derive additional conditions on $\kappa$. 
First, we compute the dimension of the $k$-vector space $(I \otimes W) \cap (W \otimes I)$. Recall that $A \otb A^\opc$ is a quadratic algebra presented as $T( W ) /(I)$. So, $(A \otb A^\opc)_3 = W^{\otimes 3} / ((I \otimes W) + (W \otimes I))$. Hence,
\[
\begin{array}{rl}
\dim W^{\otimes 3}  - \dim(A \otb A^\opc)_3 &= \dim ((I \otimes W) + (W \otimes I))\\
 &= \dim (I \otimes W) + \dim (W \otimes I) - \dim ((I \otimes W) \cap (W \otimes I)).
\end{array}
\]

\noindent Since $\dim W = 4$, $\dim I =6$, and $\dim (A \otb A^\opc)_3 = {3 + (4-1) \choose 4-1}$ (due to the PBW basis of $A \otb A^\opc$), we get that $\dim ((I \otimes W) \cap (W \otimes I)) = 4$. It may be checked that the intersection has vector space basis 
\[
\begin{array}{lll}
s_{uvu'}& := r_1u'-r_5v+r_6u &=  u'r_1-vr_5+ur_6 ,\\
s_{uvv'} &:= r_1v'+ r_3v+ r_4u &= -v'r_1- vr_3 +ur_4 + \lambda ur_5 ,\\
s_{uu'v'} &:= r_2u - r_3u' + r_5v'&= ur_2 - u'r_3 + v'r_5 ,\\
s_{vu'v'} &:= r_2v + r_4u' - r_6v' +\lambda r_5u' &= -vr_2 + u'r_4 + v'r_6 .
\end{array}
\]

\noindent Condition (c$'$) of Theorem~\ref{thm:mainWW} is thus equivalent to the following four equations: 
\[
\begin{array}{rl}
\kappa(r_1)u'-\kappa(r_5)v+\kappa(r_6)u &= u'\kappa(r_1)-v\kappa(r_5)+u\kappa(r_6),\\
\kappa(r_1)v'+\kappa(r_3)v+\kappa(r_4)u &= -v'\kappa(r_1)-v\kappa(r_3)+u\kappa(r_4) + \lambda u \kappa(r_5),\\
\kappa(r_2)u - \kappa(r_3)u' + \kappa(r_5)v' &= u\kappa(r_2) - u'\kappa(r_3) +v'\kappa(r_5),\\
\kappa(r_2)v + \kappa(r_4)u' - \kappa(r_6)v'  +\lambda \kappa(r_5)u' &= -v\kappa(r_2) + u'\kappa(r_4) + v'\kappa(r_6).
\end{array}
\]
These equations force $\gamma_1=\delta_1=\gamma_2=\delta_2=\gamma_3=\delta_3
=\beta_4=0$, so with \eqref{eq:T(2)kappa} we are left with $\kappa(r_4)=\alpha_4$ and $\kappa(r_i)=0$
for all $i\neq 4$. (For instance, $\kappa(r_1)u' = \gamma_1(g \cdot u') x+ \gamma_1(x \cdot u') 1 + \delta_1(1 \cdot u')gx + \delta_1(gx \cdot u')g = u' \kappa (r_1)$, $\kappa(r_5)v = 0 = v \kappa(r_5)$, and $\kappa(r_6) u = u \kappa(r_6)$. So, the first equation does not yield conditions on the deformation parameters.)
\end{proof}


\section{Example:  $H=k C_2$  and $A=k_J[u,v]$ (Jordan plane)}\label{sec:Jordan}

In this section, we consider the {\it Jordan plane}
$$A =k_J[u,v]:= k\langle u,v \rangle/ (r_1:= vu-uv+v^2),$$
which is well known to be Koszul. Take $k$ to be a field of odd characteristic, and take the action of the  Hopf algebra $kC_2$ on $A$, where $C_2 = \langle g ~|~ g^2 = 1 \rangle$, given by 
$$g \cdot u = -u, \quad g \cdot v = -v.$$
We compute PBW deformations of $(A \otb A^\opc) \# kC_2$ for two different braidings
 of $kC_2$-modules.


\subsection{Nontrivial braiding} \label{kJRC2}
The group algebra  $kC_2$ is quasitriangular with $\text{R}$-matrix:
\begin{equation} \label{RforC2}
\text{R}= \textstyle \frac{1}{2}(1\ot 1 + 1\ot g + g\ot 1 - g\ot g).
\end{equation}
(See, for example, \cite[Example 2.1.6]{Majid95}.) 
Let $\sf c$ be the braiding ${\sf c} = \text{R}\circ \sigma$, where $\sigma$ is the flip map.
We first determine the structure of $A^\opc$ and $A\otb A^\opc$.

Let $u',v'$ denote generators of the
braided-opposite algebra $A^\opc$. Applying the \text{R}-matrix~\eqref{RforC2},  
$$y'z' = m_A  \text{R}(z \otimes y)  = \textstyle \frac{1}{2}(zy-zy-zy-zy) = - zy,$$
for $y,z \in \{ u,v\}$. So, it is clear that $A^\opc$ is generated by $u', v'$ subject to relation $r_2:= v'u'-u'v'-v'^2$.

We claim that the relations of $A \otb A^\opc$ are
\[
\begin{array}{lll}
r_1 := vu-uv+v^2, && r_2 := v'u'-u'v'-v'^2,\\
r_3 := uu'+u'u, && r_4 := vu'+u'v,\\
r_5 := uv'+v'u, && r_6 := vv'+v'v.
\end{array}
\]
 This holds in the same fashion as the above calculation of the defining
relation for $A^\opc$. For instance, $(m_A \otimes m_{A^\opc})(1 \otimes {\sf c} \otimes 1)(1 \otimes u' \otimes v \otimes 1) = -vu'$, which yields relation $r_4$.

\begin{proposition} \label{prop:kJRC2}
Let $A=k_J[u,v]$. Then $A\otb A^\opc$ is the algebra generated by $u,v,u',v'$ subject
to the relations $r_1,\ldots, r_6$ above. 
The algebra $(A \otb A^\opc) \# kC_2$ admits a 3-parameter PBW deformation 
$$\mathcal{D}_{A \otb A^\opc, \kappa} = 
  \frac{ k\langle u,v,u',v'\rangle \# kC_2}{ (r_i-\kappa(r_i))},
$$
where $\kappa(r_1)=\lambda_1$, $\kappa(r_2)=\lambda_2$, $\kappa(r_3)=\lambda_3 g$
for $\lambda_1,\lambda_2,\lambda_3\in k$, and $\kappa(r_i)=0$  for $i=4,5,6$.  
Moreover, all PBW deformations for which $\kappa =\kappa^C$ (as assumed in Hypothesis~\ref{deg0}) are of this form.
\end{proposition}

\begin{proof}
Since $A$ is Koszul, by Corollary~\ref{cor:twistkoszul} and Proposition~\ref{lem:HopKoszul}, $A\otb A^\opc$ is Koszul, and 
we may apply Theorem~\ref{thm:mainWW} to determine PBW deformations of $(A \otb A^\opc) \# kC_2$. (Again, we do this in the case when $\kappa^L=0$.) Say $\kappa(r_i) = \alpha_i + \beta_i g \in kC_2$, for $i = 1, \dots, 6$. Note that $g \cdot r_i = r_i$, for all $i$. The $kC_2$-invariance of  $\kappa$ imposes no conditions on $\alpha_i, \beta_i$ by Lemma~\ref{lem:prelim}; indeed $Z(kC_2) = kC_2$. So Theorem~\ref{thm:mainWW}(a) is satisfied.

Now we derive conditions on $\alpha_i, \beta_i$ from the assumption that $\kappa \otimes \id = \id \otimes \kappa$ as in Theorem~\ref{thm:mainWW}(c$'$). 
Let $V$ (respectively, $V'$) be the vector space with basis $u,v$
(respectively, $u',v'$). 
Setting $W=V\oplus V'$, since $\dim W = 4$, $\dim I =6$, and $\dim (A \otb A^\opc)_3 = {3 + (4-1) \choose 4-1}= 20$ (the cardinality of the degree 3 elements in the PBW basis of $A \otb A^\opc$), we get 
$$\dim ((I \otimes W) \cap (W \otimes I)) 
~=~  \dim(I \otimes W) + \dim(W \otimes I)  - \dim W^{\otimes 3} + \dim(A \otb A^\opc)_3 ~=~ 4.$$
 The relations below  hold by applying the equation of 
Theorem~\ref{thm:mainWW}(c$'$) to 
a basis of $(I \otimes W) \cap (W \otimes I)$:
 \[
\begin{array}{rl}
\kappa(r_1)u' - \kappa(r_3)v + \kappa(r_4) u + \kappa(r_4) v &= \ 
u'\kappa(r_1) + v\kappa(r_3) - u\kappa(r_4) + v\kappa(r_4) ,\\
\kappa(r_1)v' - \kappa(r_5)v + \kappa(r_6) u + \kappa(r_6) v &= \ 
v'\kappa(r_1) + v\kappa(r_5) - u\kappa(r_6) + v\kappa(r_6) ,\\
\kappa(r_2)u - \kappa(r_3)v' + \kappa(r_5) u' - \kappa(r_5) v' &= \ 
u\kappa(r_2) + v'\kappa(r_3) - u'\kappa(r_5) - v'\kappa(r_5) ,\\
\kappa(r_2)v - \kappa(r_4)v' + \kappa(r_6) u' - \kappa(r_6) v' &= \ 
v\kappa(r_2) + v'\kappa(r_4) - u'\kappa(r_6) - v'\kappa(r_6).
\end{array}
\]
For instance,
$$r_1u'-r_3v+r_4u+r_4v ~=~vuu'-uvu'+v^2u'-uu'v-u'uv+vu'u+u'vu+vu'v+u'v^2~=~u'r_1+vr_3-ur_4+vr_4.$$
Now the  equations above are satisfied if and only if 
$\alpha_i = 0$ for $i =3,4,5,6$ and $\beta_j = 0$ for $j =1,2,4,5,6$. (For instance, the first equation implies that $\beta_1 = \alpha_3 = \alpha_4 = \beta_4 =0$.)
Thus the result holds. 
\end{proof}


\subsection{Trivial braiding} \label{kJtrivR}
We next compare with the more traditional choice of \text{R}-matrix for a group algebra, $\text{R} = 1\ot 1$, corresponding to the trivial braiding ${\sf c}=\sigma$.
Let  $A$ be the Jordan plane as before.

The braided-opposite algebra $A^\opc$, generated by $u'$ and $v'$, is just the
opposite algebra $A^{\op}$, which is generated by $u', v'$ subject to relation $r_2:= v'u'-u'v'-v'^2$.

The braided product $A\otb A^{\op}$ is the ordinary tensor product $A\ot A^{\op}$, and
the relations of $A \ot A^{\op}$ are
\[
\begin{array}{lll}
r_1 := vu-uv+v^2, && r_2 := v'u'-u'v'-v'^2,\\
r_3 := uu'-u'u, && r_4 := vu'-u'v,\\
r_5 := uv'-v'u, && r_6 := vv'-v'v.
\end{array}
\]

\begin{proposition} \label{prop:kJtrivR} 
Let $A=k_J[u,v]$. Then $A\ot A^{\op}$ is the algebra generated by $u,v,u',v'$ subject
to the relations $r_1,\ldots, r_6$ above. 
The algebra $(A \ot A^{\op}) \# kC_2$ admits a 3-parameter PBW deformation 
$$\mathcal{D}_{A \otb A^\opc, \kappa} =
   \frac{ k\langle u,v,u',v'\rangle \# kC_2}{(r_i-\kappa(r_i))},
$$
where $\kappa(r_i)=\alpha_i$ for $i=1,2,3$, and $\kappa(r_j) = 0$ for $j = 4,5,6$. 
Moreover, all PBW deformations for which $\kappa =\kappa^C$ (as assumed in Hypothesis~\ref{deg0}) are of this form. 
\end{proposition}

\begin{proof}
As in the proof of Proposition~\ref{prop:kJRC2},  Theorem~\ref{thm:mainWW}(a) is automatically satisfied. To apply Theorem~\ref{thm:mainWW}(c$'$), we need to alter the 4 equations derived from the basis of $(I \otimes W) \cap (W \otimes I)$ as follows: 
 \[
\begin{array}{rl}
\kappa(r_1)u' + \kappa(r_3)v - \kappa(r_4) u - \kappa(r_4) v &= \ 
u'\kappa(r_1) + v\kappa(r_3) - u\kappa(r_4) + v\kappa(r_4) ,\\
\kappa(r_1)v' + \kappa(r_5)v - \kappa(r_6) u - \kappa(r_6) v &= \ 
v'\kappa(r_1) + v\kappa(r_5) - u\kappa(r_6) + v\kappa(r_6) ,\\
\kappa(r_2)u - \kappa(r_3)v' + \kappa(r_5) u' - \kappa(r_5) v' &= \ 
u\kappa(r_2) - v'\kappa(r_3) + u'\kappa(r_5) + v'\kappa(r_5) ,\\
\kappa(r_2)v - \kappa(r_4)v' + \kappa(r_6) u' - \kappa(r_6) v' &= \ 
v\kappa(r_2) - v'\kappa(r_4) + u'\kappa(r_6) + v'\kappa(r_6).
\end{array}
\]
These equations are satisfied if and only if 
$\alpha_i = 0$ for $i =4,5,6$ and $\beta_j = 0$ for $j =1,\dots,6$.
Thus the result holds. 
\end{proof}


\section{Example:  $H=kC_2$  and $A=S(a,b,c)$ (3-dimensional Sklyanin algebra)}\label{sec:Sk} 

In this section, take $k$ to be a field of characteristic not equal to 2 or 3. Consider the algebra:
$$S:=S(a,b,c) = k\langle u,v,w \rangle/(r_1:=auv+bvu+cw^2, ~~r_2:=avw+bwv+cu^2,~~ r_3:=awu+buw+cv^2),$$
for $a,b,c \in k$.
If $abc\neq 0$ and $a,b,c$ are not all third roots of unity, then $S(a,b,c)$ is a {\it 3-dimensional Sklyanin algebra}, and in this case, $S$ is Koszul by a result of J. Zhang \cite[Theorem~5.11]{Smith}.

We consider the action of the  Hopf algebra $kC_2$ on $S$, where $C_2 = \langle g ~|~ g^2 = 1 \rangle$, given by 
$$g \cdot u = -u, \quad g \cdot v = -v, \quad g \cdot w =-w.$$
We compute PBW deformations of $(S \otb S^\opc) \# kC_2$ with respect to the two  $\text{R}$-matrices of $kC_2$ given in the previous section. 


\subsection{Nontrivial braiding} \label{sec:Ska}

Let $\text{R}$ be the $\text{R}$-matrix ~\eqref{RforC2} for $kC_2$ so that ${\sf c}=\text{R}\circ\sigma$,
where $\sigma$ is the flip map. 
The braided-opposite algebra $S^\opc$ is generated by copies $u', v', w'$ of $u,v,w$. Products of generators may be computed as follows: 
$$y'z' = m_S  \text{R}(z \otimes y)  = \textstyle \frac{1}{2}(zy-zy-zy-zy) = - zy,$$
for any $y,z \in  \{u,v,w\}$. So, for example, $a v'u' + b u'v' + c w'^2 = -auv - bvu - cw^2 =0$.
Further calculations show 
that $S^\opc$ is generated by $u', v', w'$ subject to relations $r_4$, $r_5$, $r_6$ below. In other words, 
$$S^\opc ~=~ S(b,a,c).$$

 The relations of $S \otb S^\opc$ are
\[
\begin{array}{lllll}
r_1:=auv+bvu+cw^2, &&r_2:=avw+bwv+cu^2, && r_3:=awu+buw+cv^2,\\
r_4:=bu'v'+av'u'+cw'^2, &&r_5:=bv'w'+aw'v'+cu'^2, && r_6:=bw'u'+au'w'+cv'^2,\\
r_7:=uu'+u'u, &&r_8:=vu'+u'v, && r_9:=wu'+u'w,\\
r_{10}:=uv'+v'u, &&r_{11}:=vv'+v'v, && r_{12}:=wv'+v'w,\\
r_{13}:=uw'+w'u, &&r_{14}:=vw'+w'v, && r_{15}:=ww'+w'w.
\end{array}
\]
For instance,  $(m_S \otimes m_{S^\opc})(1 \otimes {\sf c} \otimes 1)(1 \otimes u' \otimes w \otimes 1) = -wu'$, which yields relation $r_9$.

\begin{proposition} \label{prop:Ska} 
Let $S=S(a,b,c)$ be the Sklyanin algebra defined above.
Then $S\otb S^\opc$ is the algebra generated by $u,v,w,u',v',w'$ subject
to the relations $r_1,\ldots, r_{15}$ above. 
The smash product $(S \otb S^\opc) \# kC_2$ admits a 6-parameter PBW deformation $\mathcal{D}_{S \otb S^\opc, \kappa}$  if $a\neq b$,  and a 15-parameter PBW deformation  $\mathcal{D}_{S \otb S^\opc, \kappa}$ if $a=b$. In particular,
\begin{enumerate}
\item If $a \neq b$, then $\kappa(r_i) = \alpha_i \in k$ for $i = 1, \dots, 6$, and $\kappa(r_j) = 0$ for $j = 7, \dots, 15$.
\item If $a = b$, then $\kappa(r_i) = \alpha_i \in k$ for $i = 1, \dots, 15$. 
\end{enumerate}
Moreover, all PBW deformations for which $\kappa =\kappa^C$ (as assumed in Hypothesis~\ref{deg0}) are of this form. 
\end{proposition}

\begin{proof}
Since $S$ is Koszul, we may apply Proposition~\ref{lem:HopKoszul}, Corollary~\ref{cor:twistkoszul}, and then apply Theorem~\ref{thm:mainWW} to compute PBW deformations of $(S \otb S^\opc) \# kC_2$. (Again, we do this in the case when $\kappa^L=0$, so that $\kappa=\kappa^C$.) Say $\kappa(r_i) = \alpha_i + \beta_i g \in kC_2$, for $i = 1, \dots, 15$. Since $g \cdot r_i = r_i$, for all $i$, the $kC_2$-invariance of  $\kappa$ imposes no conditions on $\alpha_i, \beta_i$ by Lemma~\ref{lem:prelim}; indeed $Z(kC_2) = kC_2$. So Theorem~\ref{thm:mainWW}(a) is satisfied.

Now we derive conditions on $\alpha_i, \beta_i$ from the assumption that $\kappa \otimes \id = \id \otimes \kappa$ as in  Theorem~\ref{thm:mainWW}(c$'$). Let $V$ be the vector space with basis $u,v,w$, let $V'$ be a copy of $V$ with basis $u',v',w'$, and $W=V\oplus V'$. 
Since $\dim W = 6$, $\dim I =15$, and $\dim (S \otb S^\opc)_3 = {3 + (6-1) \choose 6-1}= 56$ (due to the PBW basis of $S \otb S^\opc$), we get 
$$\dim ((I \otimes W) \cap (W \otimes I)) 
~=~  \dim(I \otimes W) + \dim(W \otimes I)  - \dim W^{\otimes 3} + \dim(S \otb S^\opc)_3 ~=~ 20.$$
 The following relations are derived from a basis of $(I \otimes W) \cap (W \otimes I)$ (we keep track of the basis elements by the indeterminates of the left column):
\[
\hspace{-.15in}
{\small
\begin{array}{rrl}
\medskip
(u,v,w): &  
\kappa(r_1)w + \kappa(r_2)u + \kappa(r_3)v &
= w\kappa(r_1) + u\kappa(r_2) + v\kappa(r_3) \\
\medskip
(u,v,u'):  &
\kappa(r_1)u' + a \kappa(r_7)v + b \kappa(r_8)u + c \kappa(r_9) w &
= u'\kappa(r_1) + b v\kappa(r_7) + a u \kappa(r_8) + c w \kappa(r_9)\\
\medskip
(u,v,v'):  &
\kappa(r_1)v' + a \kappa(r_{10})v + b \kappa(r_{11})u + c \kappa(r_{12}) w &
= v'\kappa(r_1) + b v\kappa(r_{10}) + a u \kappa(r_{11}) + c w \kappa(r_{12}) \\
\medskip
(u,v,w'): &
\kappa(r_1)w' + a \kappa(r_{13})v + b \kappa(r_{14})u + c \kappa(r_{15}) w &
= w'\kappa(r_1) + b v\kappa(r_{13}) + a u \kappa(r_{14}) + c w \kappa(r_{15})
\end{array}
}
\]
\[
\hspace{-.15in}
{\small
\begin{array}{rrl}
\medskip
(u,w,u'): &
\kappa(r_3)u' + a \kappa(r_9)u + b \kappa(r_7)w + c \kappa(r_8) v &
= u'\kappa(r_3) + aw\kappa(r_7) + bu \kappa(r_9) + c v \kappa(r_8) \\
\medskip
(u,w,v'):  &
\kappa(r_3)v' + a \kappa(r_{12})u + b \kappa(r_{10})w + c \kappa(r_{11}) v &
= v'\kappa(r_3) + aw\kappa(r_{10}) + bu\kappa(r_{12}) + c v \kappa(r_{11}) \\
\medskip
(u,w,w'): &
\kappa(r_3)w' + a \kappa(r_{15})u + b \kappa(r_{13})w + c \kappa(r_{14}) v &
= w'\kappa(r_3) + aw\kappa(r_{13}) + bu \kappa(r_{15}) + c v \kappa(r_{14}) \\
\medskip
(u,u',v'):  &
\kappa(r_4)u + b \kappa(r_7)v' + a \kappa(r_{10})u' + c \kappa(r_{13}) w' &
= u \kappa(r_4) + b u'\kappa(r_{10}) + av' \kappa(r_7) + c w' \kappa(r_{13}) 
\end{array}
}
\]
\[
\hspace{-.15in}
{\small
\begin{array}{rrl}
\medskip
(u,u',w'):  &
\kappa(r_6)u + b \kappa(r_{13})u' + a \kappa(r_7) w'+ c \kappa(r_{10}) v' &
= u\kappa(r_6) + bw'\kappa(r_7) + au' \kappa(r_{13}) + cv' \kappa(r_{10}) \\
\medskip
(u,v',w'):  & 
\kappa(r_5)u + b \kappa(r_{10})w' + a \kappa(r_{13})v' + c \kappa(r_7) u' &
= u\kappa(r_5) + bv'\kappa(r_{13}) + aw' \kappa(r_{10}) + cu' \kappa(r_7)\\
\medskip
(v,w,u'):  &
\kappa(r_2)u' + a \kappa(r_8)w + b \kappa(r_9)v + c \kappa(r_7) u &
= u'\kappa(r_2) + bw\kappa(r_8) + av \kappa(r_9) + cu \kappa(r_7) \\
\medskip
(v,w,v'):  & 
\kappa(r_2)v' + a \kappa(r_{11})w + b \kappa(r_{12})v + c \kappa(r_{10}) u &
= v'\kappa(r_2) + bw\kappa(r_{11}) + av \kappa(r_{12}) + cu \kappa(r_{10})
\end{array}
}
\]
\[
\hspace{-.15in}
{\small
\begin{array}{rrl}
\medskip
(v,w,w'):  &
\kappa(r_2)w' + a \kappa(r_{14})w + b \kappa(r_{15})v + c \kappa(r_{13}) u &
= w'\kappa(r_2) + bw\kappa(r_{14}) + av \kappa(r_{15}) + cu \kappa(r_{13}) \\
\medskip
(v,u',v'):  & 
\kappa(r_4)v + b \kappa(r_8)v' + a \kappa(r_{11})u' + c \kappa(r_{14}) w' &
= v\kappa(r_4) + bu'\kappa(r_{11}) + av' \kappa(r_8) + cw' \kappa(r_{14})\\
\medskip
(v,u',w'):  &
\kappa(r_6)v+ b \kappa(r_{14})u' + a \kappa(r_8)w' + c \kappa(r_{11}) v' &
= v\kappa(r_6) + bw'\kappa(r_8) + au' \kappa(r_{14}) + cv' \kappa(r_{11}) \\
\medskip
(v,v',w'):  & 
\kappa(r_5)v + b \kappa(r_{11})w' + a \kappa(r_{14})v' + c \kappa(r_8) u' &
= v\kappa(r_5) + bv'\kappa(r_{14}) + aw'\kappa(r_{11}) + cu' \kappa(r_8)
\end{array}
}
\]
\[
\hspace{-.15in}
{\small
\begin{array}{rrl}
\medskip
(w,u',v'):  &
\kappa(r_4)w + b\kappa(r_9)v' + a \kappa(r_{12})u' + c \kappa(r_{15}) w' &
= w\kappa(r_4) + bu'\kappa(r_{12}) + av' \kappa(r_9) + c w' \kappa(r_{15}) \\
\medskip
(w,u',w'): & 
\kappa(r_6)w + b \kappa(r_{15})u' + a \kappa(r_9)w' + c \kappa(r_{12}) v' &
= w\kappa(r_6) + bw'\kappa(r_9) + au' \kappa(r_{15}) + cv' \kappa(r_{12}) \\
\medskip
(w,v',w'): &
\kappa(r_5)w + b\kappa(r_{12})w' + a \kappa(r_{15})v' + c \kappa(r_9) u' &
= w\kappa(r_5) + bv'\kappa(r_{15}) + aw' \kappa(r_{12}) + cu' \kappa(r_9) \\
\medskip
(u',v',w'):  & 
\kappa(r_4)w' + \kappa(r_5)u' + \kappa(r_6)v' &
= w'\kappa(r_4) + u'\kappa(r_5) + v'\kappa(r_6) .
\end{array}
}
\]
To get the first equation, for instance, we use
$$r_1w+r_2u+r_3v ~=~ auvw+bvuw+cw^3+avwu+bwvu+cu^3+awuv+buwv+cv^3 ~=~ wr_1+ur_2+vr_3.$$
Now the equations above are equivalent to the conditions given, in the statement of
the proposition, on the scalars $\alpha_i,\beta_i$. For instance, the equation ($u,v,w$) yields $\beta_1 = \beta_2 = \beta_3 =0$ and the equation ($u,v,u'$) yields $\beta_1 = \beta_9 =0$, etc.
\end{proof}


\subsection{Trivial braiding} \label{sec:Skb}

Using the trivial $\text{R}$-matrix $1\otimes 1$ for $kC_2$, we obtain the following result. 
The braided-opposite algebra $S^\opc$, generated by $u', v', w'$, is just the
opposite algebra of $S$, and $S^{\op}$ is generated by $u', v', w'$ subject to relations $r_4$, $r_5$, $r_6$ below. Therefore, 
$$S^\opc ~=~ S^{\text op} ~=~ S(b,a,c).$$

 The relations of $S \ot S^{\op}$ are
\[
\begin{array}{lllll}
r_1:=auv+bvu+cw^2, &r_2:=avw+bwv+cu^2, && r_3:=awu+buw+cv^2,\\
r_4:=bu'v'+av'u'+cw'^2, &r_5:=bv'w'+aw'v'+cu'^2, && r_6:=bw'u'+au'w'+cv'^2,\\
r_7:=uu'-u'u, &r_8:=vu'-u'v, && r_9:=wu'-u'w,\\
r_{10}:=uv'-v'u, &r_{11}:=vv'-v'v, && r_{12}:=wv'-v'w,\\
r_{13}:=uw'-w'u, &r_{14}:=vw'-w'v, && r_{15}:=ww'-w'w.
\end{array}
\]

\begin{proposition} \label{prop:Skb}
Let $S=S(a,b,c)$ be the Sklyanin algebra defined above.
Then $S\ot S^{\op}$ is the algebra generated by $u,v,w,u',v',w'$ subject 
to the relations $r_1,\ldots, r_{15}$ above. 
 The smash product $(S \ot S^{\op}) \# kC_2$ admits a 6-parameter PBW deformation  $\mathcal{D}_{S \ot S^{\op}, \kappa}$  if $a\neq b$,  and a 15-parameter PBW deformation  $\mathcal{D}_{S \ot S^{\op}, \kappa}$  if $a=b$. In particular,
 \begin{enumerate}
\item If $a \neq b$, then $\kappa(r_i) = \alpha_i \in k$ for $i = 1, \dots, 6$, and $\kappa(r_j) = 0$ for $j = 7, \dots, 15$.
\item If $a = b$, then $\kappa(r_i)=\alpha_i\in k$ for $i = 1, \dots, 6$, and $\kappa(r_j) = \beta_j g$ for $j = 7, \dots, 15$ with $\beta_j \in k$.
\end{enumerate}
Moreover, all PBW deformations for which $\kappa =\kappa^C$ (as assumed in Hypothesis~\ref{deg0}) are of this form. 
\end{proposition}

\begin{proof}
As in the proof of Proposition~\ref{prop:Ska}, condition (a) of Theorem~\ref{thm:mainWW} is satisfied for any choices of $\kappa(r_i)\in kC_2$. To apply Theorem~\ref{thm:mainWW}(c$'$), we need to alter the 20 equations derived from the basis of $(I \otimes W) \cap (W \otimes I)$ in the proof of Proposition~\ref{prop:Ska}, as
follows:
\[
\hspace{-.15in}
{\small
\begin{array}{rrl}
\medskip
(u,v,w): &  
\kappa(r_1)w + \kappa(r_2)u + \kappa(r_3)v &
= w\kappa(r_1) + u\kappa(r_2) + v\kappa(r_3) \\
\medskip
(u,v,u'):  &
\kappa(r_1)u' - a \kappa(r_7)v - b \kappa(r_8)u - c \kappa(r_9) w &
= u'\kappa(r_1) + b v\kappa(r_7) + a u \kappa(r_8) + c w \kappa(r_9)\\
\medskip
(u,v,v'):  &
\kappa(r_1)v' - a \kappa(r_{10})v - b \kappa(r_{11})u - c \kappa(r_{12}) w &
= v'\kappa(r_1) + b v\kappa(r_{10}) + a u \kappa(r_{11}) + c w \kappa(r_{12}) \\
\medskip
(u,v,w'): &
\kappa(r_1)w' - a \kappa(r_{13})v - b \kappa(r_{14})u - c \kappa(r_{15}) w &
= w'\kappa(r_1) + b v\kappa(r_{13}) + a u \kappa(r_{14}) + c w \kappa(r_{15})
\end{array}
}
\]
\[
\hspace{-.15in}
{\small
\begin{array}{rrl}
\medskip
(u,w,u'): &
\kappa(r_3)u' - a \kappa(r_9)u - b \kappa(r_7)w - c \kappa(r_8) v &
= u'\kappa(r_3) + aw\kappa(r_7) + bu \kappa(r_9) + c v \kappa(r_8) \\
\medskip
(u,w,v'):  &
\kappa(r_3)v' - a \kappa(r_{12})u - b \kappa(r_{10})w - c \kappa(r_{11}) v &
= v'\kappa(r_3) + aw\kappa(r_{10}) + bu\kappa(r_{12}) + c v \kappa(r_{11}) \\
\medskip
(u,w,w'): &
\kappa(r_3)w' - a \kappa(r_{15})u - b \kappa(r_{13})w - c \kappa(r_{14}) v &
= w'\kappa(r_3) + aw\kappa(r_{13}) + bu \kappa(r_{15}) + c v \kappa(r_{14}) \\
\medskip
(u,u',v'):  &
\kappa(r_4)u + b \kappa(r_7)v' + a \kappa(r_{10})u' + c \kappa(r_{13}) w' &
= u \kappa(r_4) - b u'\kappa(r_{10}) - av' \kappa(r_7) - c w' \kappa(r_{13})
\end{array}
}
\]
\[
\hspace{-.15in}
{\small
\begin{array}{rrl}
\medskip
(u,u',w'):  &
\kappa(r_6)u + b \kappa(r_{13})u' + a \kappa(r_7) w'+ c \kappa(r_{10}) v' &
= u\kappa(r_6) - bw'\kappa(r_7) - au' \kappa(r_{13}) - cv' \kappa(r_{10}) \\
\medskip
(u,v',w'):  & 
\kappa(r_5)u + b \kappa(r_{10})w' + a \kappa(r_{13})v' + c \kappa(r_7) u' &
= u\kappa(r_5) - bv'\kappa(r_{13}) - aw' \kappa(r_{10}) - cu' \kappa(r_7)\\
\medskip
(v,w,u'):  &
\kappa(r_2)u' - a \kappa(r_8)w - b \kappa(r_9)v - c \kappa(r_7) u &
= u'\kappa(r_2) + bw\kappa(r_8) + av \kappa(r_9) + cu \kappa(r_7) \\
\medskip
(v,w,v'):  & 
\kappa(r_2)v' - a \kappa(r_{11})w - b \kappa(r_{12})v - c \kappa(r_{10}) u &
= v'\kappa(r_2) + bw\kappa(r_{11}) + av \kappa(r_{12}) + cu \kappa(r_{10})
\end{array}
}
\]
\[
\hspace{-.15in}
{\small
\begin{array}{rrl}
\medskip
(v,w,w'):  &
\kappa(r_2)w' - a \kappa(r_{14})w - b \kappa(r_{15})v - c \kappa(r_{13}) u &
= w'\kappa(r_2) + bw\kappa(r_{14}) + av \kappa(r_{15}) + cu \kappa(r_{13}) \\
\medskip
(v,u',v'):  & 
\kappa(r_4)v + b \kappa(r_8)v' + a \kappa(r_{11})u' + c \kappa(r_{14}) w' &
= v\kappa(r_4) - bu'\kappa(r_{11}) - av' \kappa(r_8) - cw' \kappa(r_{14})\\
\medskip
(v,u',w'):  &
\kappa(r_6)v+ b \kappa(r_{14})u' + a \kappa(r_8)w' + c \kappa(r_{11}) v' &
= v\kappa(r_6) - bw'\kappa(r_8) - au' \kappa(r_{14}) - cv' \kappa(r_{11}) \\
\medskip
(v,v',w'):  & 
\kappa(r_5)v + b \kappa(r_{11})w' + a \kappa(r_{14})v' + c \kappa(r_8) u' &
= v\kappa(r_5) - bv'\kappa(r_{14}) - aw'\kappa(r_{11}) - cu' \kappa(r_8)
\end{array}
}
\]
\[
\hspace{-.15in}
{\small
\begin{array}{rrl}
\medskip
(w,u',v'):  &
\kappa(r_4)w + b\kappa(r_9)v' + a \kappa(r_{12})u' + c \kappa(r_{15}) w' &
= w\kappa(r_4) - bu'\kappa(r_{12}) - av' \kappa(r_9) - c w' \kappa(r_{15}) \\
\medskip
(w,u',w'): & 
\kappa(r_6)w + b \kappa(r_{15})u' + a \kappa(r_9)w' + c \kappa(r_{12}) v' &
= w\kappa(r_6) - bw'\kappa(r_9) - au' \kappa(r_{15}) - cv' \kappa(r_{12}) \\
\medskip
(w,v',w'): &
\kappa(r_5)w + b\kappa(r_{12})w' + a \kappa(r_{15})v' + c \kappa(r_9) u' &
= w\kappa(r_5) - bv'\kappa(r_{15}) - aw' \kappa(r_{12}) - cu' \kappa(r_9) \\
\medskip
(u',v',w'):  & 
\kappa(r_4)w' + \kappa(r_5)u' + \kappa(r_6)v' &
= w'\kappa(r_4) + u'\kappa(r_5) + v'\kappa(r_6) .
\end{array}
}
\]
These equations are equivalent to the conditions given, in the statement of
the proposition, on the scalars $\alpha_i,\beta_i$. 
\end{proof}

\begin{remark} \label{rem:S(11c)} More PBW deformations of $(S \otb S^\opc) \# kC_2$ were expected when $a=b$, as $S$ in this case satisfies a polynomial identity. Indeed, $S(1,1,c)$ has PI degree 2 (see, e.g. \cite[Proposition~1.6]{Wa}).  Noncommutative PI algebras typically admit more deformations and symmetries (e.g. group/ Hopf actions) than their generic counterparts.
\end{remark}


\section{Further directions} \label{directions}

In this section, we pose questions and suggest directions for future research. Here, $H$ is a Hopf algebra for which there is a  monoidal category $\mathcal{C}$ of  $H$-modules that comes equipped with a braiding ${\sf c}$; $A \otb B$ is the braided product of  two algebras $A$ and $B$ in $\mathcal{C}$; and $A^\opc$ is the braided opposite of $A$ in $\mathcal{C}$. The PBW deformation of the smash product algebra $(A \otb B) \#H$ is denoted by $\mathcal{D}:=\mathcal{D}_{A \otb B, \kappa}$ and depends on deformation parameter $\kappa$.
\smallskip

First, motivated by the Propositions~\ref{prop:kJRC2} and~\ref{prop:kJtrivR} and Propositions~\ref{prop:Ska}~ and~\ref{prop:Skb}, we ask:

\begin{question}
Is the deformation parameter space of $(A \otb B)\# H$ independent of the choice of the braiding~${\sf c}$? When $H$ is cocommutative?
\end{question}

Moreover, in this work we have focused on deformations for which $\kappa =\kappa^C$, yet one could consider:

\begin{problem} \label{kappaL}
Extend this work to classify PBW deformations of  $(A \otb B)\# H$ (or, in particular, of \linebreak $(A \otb A^\opc)\# H$) for which $\kappa^L\not\equiv 0$.
\end{problem}

\noindent To start, consider \cite[Example~4.16]{WW} for instance: there exist non-trivial PBW deformations of $k[u,v] \# T(2)$ in which $\kappa^L\not\equiv 0$ (cf. Proposition~\ref{prop:T(2)}).
\smallskip

The following problem was suggested by Pavel Etingof.

\begin{problem}  \label{prob:iCa}
Define and investigate a $q$-deformed analogue of the {\it infinitesimal Cherednik algebras} defined in work of Etingof, Gan, and Ginzburg \cite{EGG}. 
\end{problem}

 Etingof pointed out that one could start by taking $\mathfrak{g} = \mathfrak{gl}_n$ for $n \geq 2$, as the problem is settled in the case for $n=1$ (the solution modulo a central character is a quantum generalized Weyl algebra \cite{S-AV}).
 This problem also pertains to $q$-deformations of {\it continuous Cherednik algebras} \cite{EGG}. Moreover, by considering work of Losev and Tsymbaliuk \cite{LT}, the task above may yield $q$-deformations of certain W-algebras; in this direction, see Sevostyanov's work on q-W-algebras \cite{Sev}.  
At any rate, we suggest that one could use our framework of PBW deformations of braided products to attack Problem~\ref{prob:iCa} (see Remark~\ref{rk:uqgl2}).
\smallskip

Recall that 
the rational Cherednik algebras (and more generally symplectic reflection algebras) take as a starting point $A=S(V)$, a symmetric algebra, and $B= S(V^*)$, the symmetric algebra on the dual space $V^*$. 
On the other hand, in the context of braided products, we propose the following analogue of a rational Cherednik algebra. We say that an algebra $A \in \mathcal{C}$ is {\it braided commutative} if $A = A^\opc$ as algebras in $\mathcal{C}$.

\begin{definition}(braided rational Cherednik algebra)
An algebra $\mathcal{D}$ is a {\it braided rational Cherednik algebra} if $\mathcal{D}$ is a PBW deformation of $(A \otb A^\opc) \# H$, for some braided commutative algebra $A \in \mathcal{C}$. 
\end{definition}

 We ask:

\begin{question} \label{q:RCA}
What (ring-theoretic, representation-theoretic, homological) properties do braided rational Cherednik algebras share with rational Cherednik algebras?
\end{question}

More generally, one could generalize the terminology above for twisting $H$-module maps $\tau$ (instead of using braidings ${\sf c}$), and pose the same question.

\smallskip

Now prompted by work of Bazlov and Berenstein, there are a myriad of tasks one could attack. For instance,  consider the {\it minimality} of our PBW deformations of braided products (cf. \cite[Theorem~B]{BB1}). 

\begin{definition} Suppose we have presentations $A = T(V)/(I)$ and $B = T(W)/(J)$. We say that the PBW deformation $\mathcal{D}$  is {\it minimal} if $(I)$ and $(J)$ are the largest ideals of $T(V)_{>0}$ and $T(W)_{>0}$, respectively, so that $\mathcal{D}$ admits a triangular decomposition $A \otimes H \otimes B$.
 \end{definition}

\begin{problem}
Show that any PBW deformation $\mathcal{D}$ of $(A \otb B) \# H$ has a unique quotient algebra $\mathcal{D}'$ so that $\mathcal{D}'$ is a minimal PBW deformation of $(A' \otb B') \# H$, for some quotient $H$-module algebras $A', B'$ of $A, B$ in $\mathcal{C}$, respectively.
\end{problem}

\section*{Acknowledgments}
The authors would like to thank the anonymous referee for their useful comments which improved the exposition of this manuscript. We also thank Ulrich Kr\"ahmer for some very helpful correspondence
regarding algebras in braided categories, and we thank Pavel Etingof for providing useful references and for supplying Problem~\ref{prob:iCa}.
Walton was supported by NSF grants \#DMS-1550306 and 1663775.
Witherspoon was supported by NSF grants  \#DMS-1401016 and 1665286.

\end{document}